\documentclass[11pt,twoside]{article} 
\usepackage{bbm}
\usepackage{amsmath,amssymb}
\usepackage{hyperref}
\usepackage{hyperref}
\hypersetup{
    colorlinks,
    citecolor=black,
    filecolor=black,
    linkcolor=black,
    urlcolor=black}

\makeatletter
\newcommand*{\mint}[1]{%
  \mint@l{#1}{}%
}
\newcommand*{\mint@l}[2]{%
  \@ifnextchar\limits{%
    \mint@l{#1}%
  }{%
    \@ifnextchar\nolimits{%
      \mint@l{#1}%
    }{%
      \@ifnextchar\displaylimits{%
        \mint@l{#1}%
      }{%
        \mint@s{#2}{#1}%
      }%
    }%
  }%
}
\newcommand*{\mint@s}[2]{%
  \@ifnextchar_{%
    \mint@sub{#1}{#2}%
  }{%
    \@ifnextchar^{%
      \mint@sup{#1}{#2}%
    }{%
      \mint@{#1}{#2}{}{}%
    }%
  }%
}
\def\mint@sub#1#2_#3{%
  \@ifnextchar^{%
    \mint@sub@sup{#1}{#2}{#3}%
  }{%
    \mint@{#1}{#2}{#3}{}%
  }%
}
\def\mint@sup#1#2^#3{%
  \@ifnextchar_{%
    \mint@sup@sub{#1}{#2}{#3}%
  }{%
    \mint@{#1}{#2}{}{#3}%
  }%
}
\def\mint@sub@sup#1#2#3^#4{%
  \mint@{#1}{#2}{#3}{#4}%
}
\def\mint@sup@sub#1#2#3_#4{%
  \mint@{#1}{#2}{#4}{#3}%
}
\newcommand*{\mint@}[4]{%
  \mathop{}%
  \mkern-\thinmuskip
  \mathchoice{%
    \mint@@{#1}{#2}{#3}{#4}%
        \displaystyle\textstyle\scriptstyle
  }{%
    \mint@@{#1}{#2}{#3}{#4}%
        \textstyle\scriptstyle\scriptstyle
  }{%
    \mint@@{#1}{#2}{#3}{#4}%
        \scriptstyle\scriptscriptstyle\scriptscriptstyle
  }{%
    \mint@@{#1}{#2}{#3}{#4}%
        \scriptscriptstyle\scriptscriptstyle\scriptscriptstyle
  }%
  \mkern-\thinmuskip
  \int#1%
  \ifx\\#3\\\else_{#3}\fi
  \ifx\\#4\\\else^{#4}\fi
}
\newcommand*{\mint@@}[7]{%
  \begingroup
    \sbox0{$#5\int\m@th$}%
    \sbox2{$#5\int_{}\m@th$}%
    \dimen2=\wd0 %
    \let\mint@limits=#1\relax
    \ifx\mint@limits\relax
      \sbox4{$#5\int_{\kern1sp}^{\kern1sp}\m@th$}%
      \ifdim\wd4>\wd2 %
        \let\mint@limits=\nolimits
      \else
        \let\mint@limits=\limits
      \fi
    \fi
    \ifx\mint@limits\displaylimits
      \ifx#5\displaystyle
        \let\mint@limits=\limits
      \fi
    \fi
    \ifx\mint@limits\limits
      \sbox0{$#7#3\m@th$}%
      \sbox2{$#7#4\m@th$}%
      \ifdim\wd0>\dimen2 %
        \dimen2=\wd0 %
      \fi
      \ifdim\wd2>\dimen2 %
        \dimen2=\wd2 %
      \fi
    \fi
    \rlap{%
      $#5%
        \vcenter{%
          \hbox to\dimen2{%
            \hss
            $#6{#2}\m@th$%
            \hss
          }%
        }%
      $%
    }%
  \endgroup
}
\usepackage{mathrsfs}
\usepackage{amssymb}
\usepackage{amsmath}
\usepackage{amsthm}
\usepackage{amsfonts}
\usepackage{color}
\usepackage{graphicx}
\usepackage[active]{srcltx}
\usepackage{tikz}
\usepackage{pgflibraryarrows}
\usepackage{pgflibrarysnakes}

\usepackage[cp1252]{inputenc}

\usepackage{mathrsfs}
\usepackage{graphicx}

\usepackage[active]{srcltx}

\allowdisplaybreaks

\usepackage{titletoc}
\titlecontents{section}[0pt]{\addvspace{2pt}\filright}
              {\contentspush{\thecontentslabel\ }}
              {}{\titlerule*[8pt]{.}\contentspage}

\def\rr{{\mathbb R}}
\def\rn{{{\rr}^n}}

\def\fz{\infty}
\def\az{\alpha}

\def\dist{{\mathop\mathrm{\,dist\,}}}
\def\loc{{\mathop\mathrm{\,loc\,}}}

\def\lip{{\mathop\mathrm{\,Lip}}}

\def\lz{\lambda}
\def\dz{\delta}
\def\bdz{\Delta}
\def\ez{\epsilon}

\def\gz{{\gamma}}

\def\tz{\theta}

\def\osc{{\textrm{osc}}}

\def\bint{{\ifinner\rlap{\bf\kern.35em--}
\int\else\rlap{\bf\kern.45em--}\int\fi}\ignorespaces}

\def\bbint{{\ifinner\rlap{\bf\kern.35em--}
\hspace{0.078cm}\int\else\rlap{\bf\kern.45em--}\int\fi}\ignorespaces}

\def\osc{ \mathop \mathrm{\, osc\,} }

\def\diam{{\mathop\mathrm{\,diam\,}}}

\newtheorem{thm}{Theorem}[section]
\newtheorem{lem}[thm]{Lemma}
\newtheorem{rem}[thm]{Remark}
\newtheorem{cor}[thm]{Corollary}
\newtheorem{defn}[thm]{Definition}
\numberwithin{equation}{section}

\title
{\Large\bf    Regularity from $p$-harmonic potentials to $\fz$-harmonic potentials\\ in convex rings
\footnotetext{\hspace{-0.35cm}
\noindent{2020 {\it Mathematics Subject Classification:}} 35J60, 35J65, 35J70, 49N60
\endgraf
\noindent{\it Key words and phases:}  $p$-harmonic potential,  $\fz$-harmonic potential,
convex rings.
\endgraf
}}

\author{Fa Peng, Yi Ru-Ya Zhang  and Yuan Zhou}
\begin{document}


\arraycolsep=1pt
\allowdisplaybreaks
 \maketitle

\begin{center}
\begin{minipage}{13.5cm}\small
 \noindent{\bf Abstract.}\quad
The exploration of shape metamorphism, surface reconstruction, and image interpolation raises fundamental inquiries concerning the $C^1$ and higher-order regularity of $\infty$-harmonic potentials --- a specialized category of $\infty$-harmonic functions. Additionally, it prompts questions regarding their corresponding approximations using $p$-harmonic potentials. It is worth noting that establishing $C^1$ and higher-order regularity for $\infty$-harmonic functions remains a central concern within the realm of $\infty$-Laplace equations and $L^\infty$-variational problems.

\quad\quad In this study, we investigate the regularity properties from $p$-harmonic potentials $ u_p$ to $\infty$-harmonic potentials $u$ within arbitrary convex rings $\Omega=\Omega_0\backslash
\overline \Omega_1$ in $\mathbb R^n$. Here $\Omega_0$ is a bounded convex domain in $\mathbb R^n$
and $\overline\Omega_1\subset \Omega_0$ is a compact convex set. Our main results can be summarized as follows:
\begin{itemize}

\item We establish interior $C^1$ regularity for $\infty$-harmonic potentials, providing {  their $C^1$-approximation  by}  $p$-harmonic potentials; it answers an open problem by Lindgren and Lindqvist \cite{ll19,ll21}. We also prove the existence of streamlines.

\item For any real value of $\alpha$, we show that $|Du|^\alpha$ belongs to  $W^{1,2}_\loc(\Omega)$. Furthermore, we prove weak convergence of $D|Du_p|^\alpha$ to $D|Du|^\alpha$ in $L^2_{\loc}(\Omega)$ as $p\to\infty$, along with  $D|Du| ^\alpha \cdot \frac{Du}{|Du|}=0$ almost everywhere.

\item In the degenerate case where $\overline\Omega_1$ reduces to a single point, we establish the following equivalence

$$ \mbox{$\Omega_0=B(x_0,r)$ for some $r>0$ $\Leftrightarrow$ $u\in C^2(\Omega)$ $\Leftrightarrow$   $u$ is concave. }$$

\item We demonstrate that the distributional second-order derivatives $\mathcal D^2u$ are Radon measures with suitable upper bounds. Convergence of $D^2 u_p$ to $\mathcal D^2u$ weakly in a measure-theoretic sense is also established.

\end{itemize}

These results extend some known  findings \cite{kzz,ll19,ll21,s05,swy08} in two dimensions. Moreover, in planar convex rings,  we prove that 
$\infty$-harmonic potentials are twice differentiable almost everywhere, providing optimal results in this context. The second-order derivatives contribute to the absolutely continuous part of $\mathcal D^2u$, enabling

$$u(x)=\frac{1}{2}\left(\max_{\overline {B(x,\epsilon)}}u+\min_{
\overline{B(x,\epsilon)}}u\right)+o(\epsilon^2)\quad\mbox{ for almost all $x\in \Omega$ as $\epsilon\to0$}.$$

\end{minipage}
\end{center}

\tableofcontents

\section{Introduction}\label{Section1}

The $\fz$-Laplace  equation, denoted as
  $$\Delta_\fz v:=D^2vDv\cdot Dv=0\quad\mbox{in a domain $U\subset\rn$ with $n\ge2$,}$$
  is a nonlinear, highly degenerate second-order elliptic equation, particularly, one that is not in divergence form.
 In this context, we will be working with its viscosity solutions as defined by Crandall-Ishii-Lions in their work \cite{cil92}.
 These solutions are commonly referred to as $\fz$-harmonic functions in the domain $U$, and you can find the specific definitions in Section \ref{Section2} of this text.

The concept of $\fz$-harmonic functions originated with Aronsson in the 1960s \cite{ar65,ar66,ar67},
 as he was investigating the Euler-Lagrange equation for the absolute minimization of the
 $L^\fz$-functional  defined as
    $F(v,U)=\sup_U|Dv|^2$.
A function $v\in W^{1,\fz}_\loc(U) $ is considered as an absolute minimizer if the following condition holds:
 $$F(v,V)\le F(w,V)\ \mbox{whenever $V\Subset U$  and $w\in C^0(\overline V) \cap W^{1,\fz}(V)$ with $w|_{\partial V}=v|_{\partial V}$}. $$
The existence of absolute minimizers is discussed in Aronsson's work \cite{ar84}.
Jensen's research \cite{j93} provided a crucial link between $\fz$-harmonic functions and absolute minimizers. He also established their uniqueness. For alternative approaches to prove uniqueness, you can refer to the works of
 Barles-Busca \cite{bb01}, Crandall et al \cite{cgw}, Peres et al \cite{pssw09}
 and Armstrong-Smart \cite{as10}.

The $\fz$-Laplace equation is not only a crucial topic in mathematics but also finds applications in a wide range of fields, including shape metamorphism, surface reconstruction, image processing, computer vision, tug-of-war games, Lipschitz learning, and more. For specific references, you can consult works such as \cite{cepb04,cis06,cms98,cp99,ct96,ob05,pssw09}.

In shape metamorphism and surface reconstruction, the task of finding suitable reconstruction functions is of utmost importance and has diverse applications across scientific disciplines. According to the ``equal importance criteria" introduced by Cong-Parvin \cite{cp98,cp99}, these reconstruction functions are solutions to the Dirichlet problem:
\begin{align}\label{IPTu}
 \mbox{  $\bdz_\fz u=0$ \mbox{in $\Omega$};
 $u=0$ on $ \partial\Omega_0$ and $ u =1$ on $ \partial\overline\Omega_1 $.}
 \end{align}
Here, $\Omega = \Omega_0\setminus \overline\Omega_1$, with $\Omega_0$ being a bounded domain in $\mathbb{R}^n$, and $\Omega_1$ being a subdomain with its closure $\overline\Omega_1$ residing within $\Omega_0$. In some cases, $\overline\Omega_1$ may reduce to a compact connected subset of $\Omega_0$ without an interior. Such a set $\Omega$ is commonly referred to as a ring domain.
Additionally, in the context of the ``interpolation algorithm via propagation" devised by Casas-Torres \cite{ct96} for image processing, interpolation functions are also required to solve the Dirichlet problem \eqref{IPTu}. For more precise details, please refer to Section \ref{Section1.3} of the relevant literature.

Jensen's work \cite{j93} provides a significant insight into problem \eqref{IPTu}. According to Jensen, this problem has a unique viscosity solution, commonly referred to as the $\fz$-harmonic potential in $\Omega$. Throughout this paper, we will denote this unique solution as $u$.
Notably, $u$ serves as an $\fz$-harmonic function within the domain $\Omega$ and adheres to the Dirichlet boundary values $u=0$ on $\partial\overline\Omega_0$ and $u=1$ on $\partial\overline\Omega_1$. Since the Dirichlet boundary values are Lipschitz continuous, it follows that $u$ is in $C^{0,1}(\overline\Omega)$.

In the context of meeting the regularity requirements for reconstruction functions in shape metamorphism and surface reconstruction, as well as the regularity needed for interpolation functions in the interpolation algorithm via propagation, several questions arise regarding the regularity of $\fz$-harmonic potentials. These questions, detailed in Section \ref{Section1.3} for further motivation, can be summarized as follows:

\begin{itemize}
  \item[]{\hspace{-0.6cm}\bf Question 1.1.} Can the $\fz$-harmonic potential $u$ in $\Omega$ be shown to have $C^1$ regularity, meaning that $u\in C^1(\Omega)$?

  \item[]{\hspace{-0.6cm}\bf Question 1.2.}
Does the length of the (partial) derivative of $u$, denoted as $|Du|$, exhibit some Sobolev regularity? Specifically, does $|Du|$ have zero partial derivative along the direction of $Du$ within $\Omega$, as suggested by the equation?

  \item[]{\hspace{-0.6cm}\bf Question 1.3.} Is it valid to assert that $u$ is twice differentiable almost everywhere? Furthermore, what regularity can be attributed to the distributional second-order derivatives of $u$?
\end{itemize}

These questions delve into the regularity properties of $\fz$-harmonic potentials and their derivatives, which are crucial in various applications. Further insights into their motivations can be found in Section \ref{Section1.3}.

The study of the regularity of $\fz$-harmonic potentials, especially considering Questions 1.1-1.3, has inherent significance. It not only addresses important issues but also provides insights into the potential $C^1$ and higher-order regularity of $\fz$-harmonic functions, which is a central topic in this field. In particular,
one of the long-standing conjectures is as follows:
\begin{equation}\label{c1IF}
\text{Conjecture: $\fz$-harmonic functions are continuously differentiable, i.e., they are $C^1$.}
\end{equation}
Notably, thanks to the $\fz$-harmonic function $x^{4/3}-y^{4/3}$ discovered by Aronsson \cite{ar84}, it is natural to inquire whether $\fz$-harmonic functions exhibit $C^{1,1/3}$ regularity and Sobolev $W^{2,\gz}$ regularity with $\gz<3/2$. Several significant contributions in the literature have addressed these questions, including works by Crandall-Evans \cite{ce01}, Savin \cite{s05}, Evans-Savin \cite{es08}, Evans-Smart \cite{es11a,es11b}, Koch wtih two of the authors \cite{kzz}, and Dong and the authors \cite{dpzz}, among others.

Specifically, in two dimensions, Savin \cite{s05} proved that $\fz$-harmonic functions are $C^1$, thus confirming the conjecture; see also \cite{zz20} for a simpification  of the original proof via capacity. Evans-Savin \cite{es08} demonstrated that $\fz$-harmonic functions have $C^{1,\gz}$ regularity for some $0<\gz<1/3$. Furthermore, in \cite{kzz} and \cite{dpzz}, it was shown that, for any $\alpha>0$, $|Dv|^\alpha$ belongs to $W^{1,2}_\text{loc}$, $D |Dv|^\alpha\cdot Dv=0$ almost everywhere, and additionally, $-\det D[|Dv|^\alpha Dv]$ is a nonnegative Radon measure.

In dimensions equal to or greater than three, Evans-Smart \cite{es11a,es11b} established that $\fz$-harmonic functions are everywhere differentiable. These contributions collectively advance our understanding of the regularity properties of $\fz$-harmonic functions and their derivatives.

The established regularity results for $\fz$-harmonic functions naturally extend to $\fz$-harmonic potentials. Additionally, in the case of two-dimensional space, when $\overline \Omega_1$ reduces to a single point within $\Omega_0$, further results have been obtained:
By Savin-Wang-Yu \cite[Corollary 1.2]{swy08}, it has been shown that $u$ belongs to $C^2(\Omega)$ if and only if $\Omega$ is of the form $\Omega=B(x_0, r)$ for some $r>0$.
Also in the work of Lindgren-Lindqvist \cite[Theorem 4]{ll19}, it was demonstrated that $u$ does not belong to $C^{1,1}_{\text{loc}}(\Omega)$ if $\Omega_0$ is not a disk.
By Brustad \cite{b22},  the $\fz$-harmonic potential in $[-1,1]^2\setminus\{0\}$  is not everywhere twice differentiable. 
These findings provide valuable insights into the regularity of $\fz$-harmonic potentials in specific scenarios.

On the other hand, an interesting aspect of $\fz$-harmonic functions is their approximation in $C^{0,\gamma}$ and weakly in $W^{1,q}_\text{loc}$ by $p$-harmonic functions that share the same Dirichlet boundary value. This approximation property has been explored by Bhattacharya-DiBenedetto-Manfredi \cite{bdbm}, as well as in the works of \cite{j96,lm95}.

A function $v\in W^{1,p}(U)$ is termed $p$-harmonic if it serves as a weak solution to the $p$-Laplace equation:
 $$\Delta_pv:={\rm div}(|Dv|^{p-2}Dv)=0\quad\mbox{in $U$.}$$
In the quest for achieving $C^1$-regularity or higher-order regularity for $\fz$-harmonic functions (and, in particular, addressing Conjecture \eqref{c1IF}), a natural idea is to establish $C^1$-regularity or higher-order regularity for $p$-harmonic functions uniformly across all large values of $p$. This would effectively approximate $\fz$-harmonic functions with $p$-harmonic functions. However, this endeavor presents significant challenges. While $p$-harmonic functions have been shown to possess $C^{1,\alpha}$ regularity and some higher-order regularity, these results are heavily dependent on the specific value of $p$ and are not uniform for all large values of $p$. (References: \cite{ev82,im89,lu68,mw88,u77}, among others).

Moreover, in two-dimensional space ($n=2$), although Savin \cite{s05} has demonstrated that $\fz$-harmonic functions are $C^1$, the problem of approximating $\fz$-harmonic functions $v$ in $C^1$ by $p$-harmonic functions $v_p$ with the same Dirichlet boundary value, or equivalently, showing that $v_p$ has $C^1$-regularity uniformly across all large values of $p$, remains a challenging open question. The most recent developments, as presented in \cite{ll21,dpzz}, aided by \cite{kzz}, have shown that for any $\alpha>0$, $|Dv_p|^\alpha$ converges to $|Dv|^\alpha$ in $L^2_\text{loc}$, and $\det D[|Dv_p|^\alpha Dv_p]$ weakly converges to $\det D[|Dv|^\alpha Dv]$ in the sense of measures. These findings represent significant progress toward understanding the regularity of $\fz$-harmonic functions and their approximation by $p$-harmonic functions.

To initiate our exploration, we seek to gain insights into the $C^1$ and higher-order approximation questions for a specific class of $\fz$-harmonic functions, such as the $\fz$-harmonic potential $u$ within $\Omega$. For this purpose, let us consider $2 < p < \fz$ and introduce $u_p$ as a $p$-harmonic potential within $\Omega$. $u_p$ is defined as follows:

$u_p$ is a unique $p$-harmonic function in $\Omega$ with Dirichlet boundary values $u_p=i$ on $\partial\Omega_i$. In other words, it is the unique weak solution to the Dirichlet problem:
\begin{align}\label{IPTp}
 \mbox{  $\bdz_p u=0$ \mbox{in $\Omega$};
 $u_p=1$ on $ \partial\Omega_1$ and $ u_p =0$ on $ \partial \Omega_0$.}
 \end{align}
Based on the findings of \cite{bdbm}, we know that $u_p$ converges to $u$ in $C^{0,\gamma}(\overline\Omega)$ with $\gamma\in (0,1)$.

In the context of the previous discussions and the motivation provided, we can now address the following inquiries related to the approximation by $p$-harmonic functions.

\begin{itemize}
\item[] {\hspace{-0.6cm}\bf Question 1.1A.}  Is it accurate to assert that $u_p$ converges to $u$ in $C^1(\Omega)$? Equivalently, does $u_p$ belong to $C^1(\Omega)$ uniformly for all values of $p$?

\item[] {\hspace{-0.6cm}\bf Question 1.2A.}   Does $|Du_p|$ exhibit Sobolev regularity that is uniform across different values of $p$? Alternatively, does $|Du_p|$ converge to $|Du|$ within a certain Sobolev space?

\item[] {\hspace{-0.6cm}\bf Question 1.3A.} Can we establish that $u_p$ converges to $u$ in specific function spaces with second-order regularity? Moreover, what kind of regularity does the second-order derivative $D^2u_p$ possess, and is this regularity uniform across a range of values for $p>2$?
\end{itemize}
These questions aim to explore the regularity properties of $p$-harmonic functions as they approximate $\fz$-harmonic functions, providing insights into the convergence and regularity behavior in the context of different values of $p$.

In the two-dimensional space, i.e.\ $n=2$, some progress has been made toward addressing Questions 1.2A and 1.3A as in \cite{ll21,dpzz}.
Furthermore, in the specific case where both $\Omega_0$ and $\overline\Omega_1$ are convex, Lindgren-Lindqvist \cite{ll19,ll21} have made significant strides. They have demonstrated that $|Du_p|$ converges to $|Du|$ locally uniformly. However, the question of whether $Du_p$ converges to $Du$ locally and uniformly remains an open question, referred to as Question 1A. These developments represent substantial advancements in our understanding of the regularity and convergence behavior of $p$-harmonic functions with respect to $\fz$-harmonic functions in two dimensions.

\subsection{Main results: regularity and  approximation  in convex rings}\label{Section1.1}

The primary objective of this paper is to offer partial answers to Questions 1.1-1.3 by addressing, in part, Questions 1.1A-1.3A within the context of arbitrary convex ring domains in dimensions $n\ge 2$.

We say
   $\Omega$ is a  convex ring   in $\rn$ with $n\ge2$ if
  $\Omega=\Omega_0\setminus \overline \Omega_1$ is a ring domain, and both of $\Omega_0$ and $\overline \Omega_1$ are convex.  Here $\overline \Omega_1$ is a compact subset of $\Omega_0$.
Denote by $u $ the   $\fz$-harmonic potential   in $\Omega$, and for $p\in(2,\fz)$, by  $u_p$   the  $p$-harmonic potential   in $\Omega$.
  We  extend $u$ and $u_p$ continuously to whole $\Omega_0$ by setting $u=u_p=1$ in $\overline\Omega_1$.

Our initial result is a comprehensive answer to Question 1.1 and Question 1A.

 \begin{thm}\label{c1-re}
We have  $u\in C^1(\Omega)$ and $|Du|\ne 0$ in $\Omega$.
Moreover $u_p\to u$ in  $C^1(\Omega)$  as $p\to\fz$.

Consequently, for each $t\in(0,1)$ the level set $\{x\in\Omega|u(x)=t\}$ is a $C^1$-regular
$(n-1)$-dimensional surface.
\end{thm}

In dimension $n=2$,   $u\in C^1(\Omega)$ can be attributed to the remarkable result on the $C^1$ regularity of $\fz$-harmonic functions established by Savin \cite{s05}.
Notably, Theorem \ref{c1-re} provides a significant outcome, demonstrating that $Du_p\to Du$ uniformly in each compact subset of $\Omega$. This result effectively resolves the previously mentioned open question raised by Lindgren-Lindqvist   \cite{ll19,ll21}.  It's worth noting that one should not expect the uniform convergence of
 $Du_p\to Du$  accross entire domain $\Omega$, as witted by the example in \cite{ll21}:  In  $\Omega=B(0,1)\setminus\{0\}\subset \rr^2$,
the $\fz$-harmonic potential  $u(x)=1-|x|$ and  the $p$-harmonic potential
 $u_p(x)=1-|x|^{\frac{p-2}{p-1}}$ with $p>2$. Obviously,  the convergence $Du_p\to Du$
is not uniform in $\Omega$.

Theorem \ref{c1-re} allows us to  establish   the following existence of trajectory (streamline).

\begin{cor} \label{streamline}
Given a point $x\in\Omega$,
there is a curve $\gz_x\in C^{0,1}([0,T_x];\overline\Omega)\cap C^1([0,T_x);\Omega)$  for some $T_x\in(0,\fz)$  as    a  trajectory (streamline)  of $u$  starting from $x$ and terminating at $\partial \Omega_1$, that is,
\begin{align}\label{ex-eq}
\frac{d\gz_x(t)}{d t}=Du(\gz_x(t))\quad \forall t\in [0,T_x) \quad \mbox{and }\ \quad \gz_x(0)=x,
\quad\gz_x(T_x)\in \partial  \Omega_1.
\end{align}
Moreover, the function $u\circ \gz _x $ is convex  and its speed $|Du\circ \gz_x|$ is nondecreasing.
\end{cor}

In two-dimensional space, Lindgren-Lindqvist's work, as presented in \cite{ll19,ll21}, not only established the existence of streamlines, as demonstrated in Corollary \ref{streamline}, but also demonstrated their uniqueness. However, in higher dimensions $n\ge 3$, the question of the uniqueness of $\gz_x$ remains unclear and requires further investigation.

Next, we give partial answers to  Question 1.2 and Question 1.2A as below.
Here and below, by abuse of notation,  we use $|Du|^\alpha $ with $\az=0$ to  denote  $\ln |Du|$.

\begin{thm}\label{sob-re}
For any $\alpha\in\rr$,  one has $|Du|^\alpha\in W^{1,2}_\loc(\Omega)$,
$D|Du|^\alpha\cdot \frac{Du}{|Du|}=0 $ almost everywhere in $\Omega$,
and $|Du_p|^\alpha \to |Du|^\alpha$ weakly in $ W^{1,2}_\loc(\Omega)$ as $p\to\fz$.
\end{thm}

In dimension $n=2$,   Theorem \ref{sob-re} is essentially
a consequence of \cite{kzz,ll21}.  Indeed,
when $\alpha>0$,  it is a direct consequence of
  \cite{kzz,ll21}. The case where $\az\le 0$ then
follows from Theorem \ref{c1-re} as $ |Du|\ne 0$.

Finally, we concentrate on Question 1.3 and Question 1.3A.
We begin with the following equivalence in the degenerate case.

 \begin{thm}\label{th-geo}
Suppose that   $\Omega_1=\{x_0\}$ for some $x_0\in\Omega_0$.
The following are equivalent:
\begin{enumerate}
\item[(i)]  $\Omega=B(x_0,r)\setminus\{x_0\}$ for some $r>0$.
\item[(ii)]  $u\in C^2(\Omega)$.
\item[(iii)]  $u$ is concave.
\end{enumerate}

\end{thm}
In dimension $n=2$, (i)$\Leftrightarrow$(ii) was already  proved by
Savin-Wang-Yu \cite[Corollary 1.2]{swy08}, where they do not require
the convexity of   $\Omega_0$.
 In dimension $n\ge3$,  it is unclear to us whether the convexity assumption of $\Omega_0$
in Theorem \ref{th-geo} can be removed.

In a general convex ring domain $\Omega$, we establish the  following regularity for distributional second order derivatives  in a
measure sense, which
 contributes to Question 1.3 and Question 1.3A.  The distributional derivative
$\mathcal Dw=(\mathcal D_iu)_{i=1}^n$  in the sense of distribution of a function $w\in L^1_\loc(\Omega)$ is defined via integration by parts, that is,
$$\int_\Omega \mathcal D_iw(x)\phi(x)\,dx=-\int_\Omega w(x) \phi_{x_i}(x)\,dx\quad \forall\phi\in C_c^\fz(\Omega).$$

\begin{thm}\label{secsob}

(i)
  The distributional derivatives $\mathcal D^2u$ are Radon measures satisfying
$$\int_\Omega \langle \mathcal D^2u Du, Du\rangle\xi \,dx=0
 \quad \forall  \xi\in C^\fz_c(\Omega).$$
 Moreover,   $Du \in BV_\loc(\Omega)$,  that is, the distributional derivatives $\mathcal D (Du)$  are Radon measures in $\Omega$,  and   $\mathcal D (Du)=\mathcal D^2u$.

(ii) $D^2 u_p\in L^1_\loc(\Omega)$ uniformly in $p\in[4,\fz)$
and $D^2 u_p\to \mathcal D^2u$ weakly in  the sense of measure, and
 $Du_p\to Du$ weakly in $BV_\loc(\Omega)$.
\end{thm}

 Additionally, when $n=2$, we are able to prove the following  almost everywhere twice differentiability  of $u$, which also contributes to
Question 1.3 and Question 1.3A.

\begin{thm}\label{twice-re} In dimension $n=2$,  $u$ is twice differentiable and
  $Du$ are differentiable almost  everywhere in $\Omega$.
 The absolutely continuous part of the measure $\mathcal D^2u$ with respect to Lebesgue measure $dx$ is  given by  $D^2u\,dx$.

Consequently,  for almost  all  $x\in \Omega$ one has
 $-D^2uDu\cdot Du=0$,
 \begin{align}\label{Taylor-exp}
  u(x)=\frac 12 \left[u\left(x+hDu(x)
 \right)+u\left(x-hDu(x)
 \right)\right]
  +o(|h|^2) \quad \mbox{as $h\to0$,}
  \end{align}
 and
 \begin{align}\label{Taylor-exp2} u(x)=\frac{1}{2}\left(\max_{\overline{B(x,\ez)}}u+\min_{\overline{B(x,\ez)}}u\right)
 +o(\ez^2)\quad\mbox{
as $\ez\to0$.}  \end{align}
  \end{thm}

 The above almost  everywhere twice differentiability of $u$ is  optimal; indeed,
 by Brustad \cite{b22},  the $\fz$-harmonic potential in $[-1,1]^2\setminus\{0\}$  is not everywhere twice differentiable.

The concept of an $\infty$-harmonic function in a given domain $U\subset \rr^n$ is closely related to the satisfaction of equation \eqref{Taylor-exp2}.  This equivalence is detailed in the work of Parviainen-Manfredi-Rossi  \cite{pmr10}.
However, it's important to note that, in general, an $\infty$-harmonic function
 $v$ does not necessarily satisfy equation \eqref{Taylor-exp2} pointwise. This point is discussed in \cite{pmr10} as well.

The significance of Theorem \ref{twice-re} lies in its assertion that $\infty$-harmonic potentials within planar convex rings satisfy equation \eqref{Taylor-exp2} almost everywhere. This result is notable because equation \eqref{Taylor-exp2} was initially utilized by Oberman \cite{ob05} for constructing numerical approximations of $\fz$-harmonic functions. Additionally, it has roots in the Dynamic Programming Principle as described in \cite{as12,pssw09}. This highlights the practical utility and theoretical foundations of equation \eqref{Taylor-exp2} in various mathematical and computational contexts.

 \subsection{Key ideas to prove main results}\label{Section1.2}

We  are ready to sketch the ideas to our main results.
The quasi-concavity of $u_p$(that is convexity of  super level sets of $u_p$) obtained by Lewis
\cite{l77} play a  crucial role. Since $u_p\to u$ in $C^0(\Omega)$,
  the  quasi-concavity of $u$ follows;
see Section \ref{Section3}.  In Section \ref{Section2},
we recall several necessary facts about $\fz$-harmonic functions.

{\it Ideas to prove Theorem \ref{c1-re}.}
  First, for any $x\in\Omega$, denote by $N_x$ the outer normal unit vectors $\nu$ of the level set
$\{y\in \Omega:u(y)=u(x)\}$, that is,
$\nu\cdot(y-x)\le0$   whenever  $u(y)>u(x)$.
By choosing   some suitable cones and using comparison with cones from above,
 we get the following lower bound
\begin{align}\label{key1u}
 -Du(x)\cdot \nu \ge \frac{u(x)-u(x+t\nu )}{t} \quad \forall \nu\in N_x,\ \forall 0<t<\frac14{\rm dist}(x,\partial\Omega_0) ,
\end{align}
  in particular, $|Du(x)|>0$.
As a results of this,  together with the differentiability of $u$ at $x$,  we deduce that
  $ N_x $ consists of the unique vector $\nu_x:=-\frac{Du(x)}{|Du(x)|}$ and hence
  \eqref{key1u} holds with $\nu=\nu_x$; see Lemma \ref{low-u} in Section \ref{Section4}.

 Next,   thanks to \eqref{key1u}  with $\nu= \nu_x$ and  the  convexity of $\Omega_1\cup\{x\in \Omega :u(x)>u(z)\}$, we are able to obtain
  the continuity of $ \nu_x$ in $x\in\Omega$,
   that is,  $$\nu_x \to \nu_z\quad\mbox{when $x\to  z\in \Omega$.  } $$
 Using this,  \eqref{key1u}  with $ \nu_z$ and continuity of $u$ we further get
 \begin{align}\label{key2u}
\liminf_{x\to z}|Du(x)|\ge \frac1t[u(z)-u(x+t\nu_z)] \quad \forall z\in \Omega,\quad \forall 0<t<\frac 14{\rm dist}(z,\partial\Omega_0),
\end{align}
Sending $ t\to0$ we have $$\liminf_{x\to z}|Du(x)|\ge |Du(z)|.$$
Recalling $\limsup_{x\to z}|Du(x)|\le |Du(z)|$ by  Crandall-Evans-Gariepy \cite{ceg01}, one gets
the continuity of $|Du|$ at any $z\in\Omega$. Thanks to  $Du(x)=-|Du(x)|\nu_x$,
we know $Du \in C^0(\Omega)$. See Section \ref{Section5} for more details.

Observe that $u_p$ also enjoys comparison with cones from above by Lewis \cite{l77}.
Since
\eqref{key1u} holds for $u_p$ at any point $x$
 and with $\nu=-\frac{Du_{p }(x)}{|Du_{p}(x)|}$ (see Section \ref{Section4} for details),
by an argument similar to  above, we could show that $Du_{p_i}(x_i)\to Du(x)$
whenever $p_i\to\fz$ and $x_i\to x$.
 Thus $Du_p\to Du$ locally uniformly; see Section \ref{Section5}.

 {\it Ideas to prove Corollary \ref{streamline}.}  Theorem \ref{c1-re} allows us to employ the  idea from  Lindgren-Lindqvist \cite{ll19,ll21} in dimension $n=2$
  to prove Corollary \ref{streamline}. For reader's convenience, we give the details in the Appendix A.

{\it Ideas to prove Theorem \ref{sob-re}.}
 For any smooth quasi-concave function $v$, we  observe a crucial fundamental geometric  structural inequality
 \begin{align}\label{key3up} 2[| D^2v Dv|^2-\bdz v\bdz_\fz v]\ge |Dv|^2[| D^2v|^2-(\bdz v)^2]
\end{align}
and also a useful divergence formula
 \begin{align}\label{e6.x1-1}{\rm div}\left(|Dv|^{-2}(\bdz v Dv-D ^2v Dv)\right)\ge 0\quad \mbox{when $Dv\neq 0$}.\end{align}
Their proofs heavily rely  on
 the convexity of the super level sets of $u$, which  implies the
 non-positive definite of
second fundamental form of its boundary   and the  non-positivity of sectional curvatures.
See Section \ref{Section6} for  details.

Applying  the divergence formula \eqref{e6.x1-1} to $u_p$ and   testing $|Du_p|^{ \alpha}\phi^2$ for any $\alpha<0$ and any $\phi \in C^\fz_c(\Omega)$  we obtain that
$|Du_p|^\alpha\in W^{1,2}_\loc$ uniformly in  $p>2$.
Since $|Du_p|$ is bounded from above and away from zero uniformly in $p$,
from $|Du_p|\in W^{1,2}_\loc$ uniformly in  $p>2$,
we conclude $|Du_p|^\alpha\in W^{1,2}_\loc$  uniformly in $ p>2$ for any $\alpha\in \rr$.
Therefore, together with $C^1$-convergence of $u_p\to u$,   we conclude
$|Du|^\alpha\in W^{1,2}_\loc$ for all $ \alpha$.
 See Section \ref{Section7}  for  details.

{\it Ideas to prove Theorem \ref{th-geo}.}
If $ \Omega =B(x_0,r)\setminus\{x_0\}$, then the $\fz$-harmonic potential therein  is given by $ 1-\frac{|x-x_0|}r$,  which is obviously of $C^2(\Omega) $ and concave in $B(x_0,r)$.
It then suffices to show  (ii) $\Rightarrow$(i) and (iii) $\Rightarrow$(i).
The proof is given in Section \ref{Section8}.
The proof relies on the fact
\begin{align}\label{con-gl}
\lim_{x\to x_0}|Du(x)|=\frac{1}{
{\rm dist}(x_0,\partial\Omega_0)},
\end{align}
which is obtained  by using the asymptotic property of $ u$ around $x_0$ by Savin-Wang-Yu \cite{swy08} and Lemma \ref{low-u}(that is \eqref{key1u} with $\nu=\nu_x$).

If $u\in C^2$, there exists a unique
 streamlines $\gz_x$ starting from any point  $x\in \Omega$ and ending at $x_0$.
Since $u\in C^2$ implies $\Delta _\fz u=0$ everywhere,  we know that
 $|Du|$  is a constant  along  $\gz_x$, and hence by \eqref{con-gl},
  is  given by $\frac{1}{{\rm dist}(x_0,\partial\Omega_0)}.$
The speed $|\dot \gz_x|$ is also $\frac{1}{
{\rm dist}(x_0,\partial\Omega_0)}.$    By a direct calculation one has
$ |x-x_0|\le  {\rm dist}(x_0,\partial\Omega_0) $ for all $x\in\Omega$.  Thus  (ii) $\Rightarrow$(i).

  If $u$ is concave, then $u$ is twice differentiable almost everywhere.
The concavity of $u$ further show that $D^2u(x)\xi\cdot\xi\le 0$
for all $\xi\in\rn$ and almost all $x\in\Omega$. This allows us to get
$D^2uDu=0$ almost everywhere.
Since $D|Du|^2=D^2uDu$ almost everywhere and $|Du|^2\in W^{1,2}_\loc(\Omega)$,
we know that $ |Du|$ is a constant, and hence by \eqref{con-gl},
  is  given by $\frac{1}{{\rm dist}(x_0,\partial\Omega_0)}.$
With this in hand, via streamline we do prove that $\Omega=B(x_0,r)$ similarly,
that is, (iii) $\Rightarrow$(i).

 {\it Ideas to prove Theorem \ref{secsob}.}  Applying \eqref{key3up} to $u_p$ and using the equation
 $-\bdz_p u=0$, we derive that
 $$| D^2u_p|\le 2|\bdz u_p|+2| D|Du_p||\quad \forall p\in [4,\fz).$$
 Since $-\Delta u_p$ is nonnegative  and $| D|Du_p||\in L^1_\loc(\Omega)$ uniformly in $p\in[4,\fz)$, we know that  $| D^2u_p|\in L^1_\loc(\Omega)$ uniformly in $p\in[4,\fz)$.
 So $D^2u_p$ weakly converges to some Radon measure $\mu$.
Note that $\mathcal D^2u$ coincides with $\mu$ in the sense of  distribution.
Moreover, we also derived a quantitative upper bound
\begin{align}\label{key4u}
\|\mathcal D^2u\|(B(x,r)) \le 2[-\bdz u](B(x,2r))+2\int_{B(x,2r)}| D|Du||\,dx,
\end{align}
where $-\bdz u$ is a nonnegative Radon measure.
 Here the measure $\|\mu \|$ stands for the total variation
of a signed Radon measure $\mu$. If $\mu$ is nonnegative, we write
 $\|\mu\|=\mu$.

  {\it Ideas to prove Theorem \ref{twice-re}.}  To obtain almost everywhere twice differentiability
  of $u$,  it suffices to prove almost everywhere  differentiability of $Du$.
 By Rademacher's theorem \cite{r19} (see also \cite{s25}), one only needs to show the pointwise Lipschitz constant $\lip (Du)(x)<\fz$   for almost all $x\in\Omega$.

  To this end, we observe a monotonicity property of $Du$ in Appendix B, which allows us to
  bound
  $$\lip (Du)(x)\le \limsup_{r\to0}\frac{\|\mathcal D^2u\|(B(x,r))}{r^2}.$$
Considering  \eqref{key4u}, noting $ |D|Du||\in L^2_\loc(\Omega)$ implies
$$\lim_{r\to 0}\mint-_{B(x,2r)}| D|Du||\,dx<\fz\ \mbox{for almost all  $x\in \Omega$,}$$
we still need to check
$$\limsup_{r\to0}\frac{[-\Delta u](B(x,r))}{r^2}<\fz \ \mbox{for almost all $x\in \Omega$}.$$
This holds because the singular part $[-\Delta u]_s$ of  the  measure $-\Delta u$ satisfies
 $$ \frac{[-\bdz u]_s(B(x,2r)) }{r^2}\to 0 \ \mbox{for almost all $x\in \Omega$}.$$

 \subsection{Equal importance criteria, interpolation by propagation, $\fz$-harmonic potential and approximation by $p$-harmonic potentials}\label{Section1.3}

The need to reconstruct intermediate shapes that gradually transition from a given source shape to a desired target shape is a fundamental requirement in a variety of scientific fields. This process forms a critical component in the study of shape metamorphism, surface reconstruction, image interpolation, and other related areas. Relevant references include \cite{cepb04,cis06,cms98,cp98,cp99,orb18,scb03}, among others.

To be more specific, consider two distinct compact $(n-1)$-dimensional surfaces, denoted as   $S_0$ and $S_1$.
The objective is to construct a family $\{S(t)\}_{s\in[0,1]}$ of  $(n-1)$-dimensional surfaces  so that $S(0)=S_0 $, $S(1)=S_1$, and $S(t) $ maintains continuity for $t\in[0,1]$.
Additionally, for each $t\in(0,1)$, $S(t) $  is expected to exhibit a certain level of smoothness, and the entire family $\{S_t\}_{t\in[0,1]} $  is intended to span the metamorphism region
$\overline \Omega=(\overline\Omega_0\setminus  \Omega_1)\cup
(\overline\Omega_1\setminus  \Omega_0)$, where $\Omega_i$  is the    domain  enclosed by  $S_i$ for $ i=0,1$. This process is crucial for a wide range of applications, facilitating the smooth transition between different shapes while ensuring continuity and smoothness in the intermediate shapes.

The goal is to represent the desired intermediate surfaces using a reconstruction and interpolation function
 $u:\overline\Omega\to\rr
$ satisfying $\{x\in\overline\Omega|u(x)=i\} =S_i$ for $i=0,1$ and hence
the Dirichlet boundary  condition
\begin{align}\label{bv}\mbox{
 $u=1$ on $S_1=\partial\Omega_1$ and $ u =0$ on $S_0=\partial\Omega_0$,}
 \end{align}
where it is assumed that,
\begin{align}\label{omega}\mbox{$S_0\cap S_1=\emptyset$ and $\Omega\ne\emptyset$,  that is,  either $\overline \Omega_0\subset\Omega_1$ or $\overline \Omega_1\subset\Omega_0$.}
\end{align}
Without loss of generality,  we can assume that $\overline  \Omega_1\subset \Omega_0$.

For any reconstruction or interpolation function $u$  that meets the criteria of $u\in C^1(\Omega)\cap C^0(\overline \Omega)$ and  $|Du|\ne 0$ in $\Omega$,  the implicit function theorem allows us to construct a collection of level surfaces:
$$\left\{\{x\in\overline\Omega|u(x)=t\}\right\}_{t\in[0,1]}.$$
These level surfaces provide a means to select intermediate surfaces or interpolation surfaces between the images $S_0,S_1$.
One may also  view
 $$ \mathcal S=\{(x,z)\in\overline\Omega\times[0,1]|u(x)=z,z\in[0,1]\}=\cup _{t\in[0,1]}S(t)\times\{t\}$$
as a reconstructed  $n$-dimensional surface
 from cross-sectional $(n-1)$-dimensional surfaces  $S_0$ and $S_1$, that is,
 $\mathcal S\cap \rr^n\times\{i\}=S_i$ for $i=0,1$. This approach allows for the smooth construction of intermediate surfaces based on the given boundary conditions and the choice of the interpolation function.

The generation of suitable reconstruction functions and the creation of intermediate shapes that fulfill various requirements are fundamental challenges in diverse scientific fields. The literature, as found in references such as \cite{cms98,cis06,cp99,ob05}, has introduced several critical constraint principles for the reconstruction process. Correspondingly, various numerical methods have been developed and widely applied.

It's important to note that different sets of constraints can lead to the study of solutions to different partial differential equations. This demonstrates the versatility and adaptability of reconstruction techniques to address the specific needs of different scientific domains.

 In the absence of any information about the reconstruction process,
we expect that
 every point in $\Omega$ is equally important and contributes similarly to
 the reconstruction  process, and that any other assumption means that we know something about   reconstruction functions.   This  natural constraint is now known as {\it equal importance criteria}, which was originally introduced by \cite{cepb04,cms98,cp99} in the study of shape metamorphism and surface reconstruction.
It was further formulated therein by requiring that the change of the reconstruction function $u$ in the gradient-magnitude along direction  is $0$, that is,
\begin{align}\label{eic}
D|Du|\cdot \frac{Du}{|Du|}=0 \quad\mbox{in $\Omega$}.
\end{align}
 This implies, along each trajectory of the gradient of $u$, the magnitude of the gradient is  a constant;
  in another words, the height of $u$ decrease linearly from 1 to 0,
 and the level surfaces of $u$
 are equally distributed along the gradient.
Together with Dirichlet boundary condition, the {\it equal importance criteria} leads to the problem:
\begin{align}\label{eicpb}
\mbox{Find a function $u:\overline\Omega\to\rr$ to solve \eqref{eic} with  Dirichlet boundary condition \eqref{bv}.}
\end{align}
We expect that a solution $u$ to \eqref{eicpb} should enjoy a nice regularity that $u\in C^0(\overline \Omega)\cap C^1(\Omega)$,  $Du\ne 0$ and $ |Du|$ has partial derivative $0$ along the direction $Du$.

 Moreover, to find a suitable interpolation function as required by image process,
  a remarkable  {\it interpolation algorithm via  propagation} was introduced by  Casas and Torres  \cite{ct96}; see also Caselles-Morel-Sbert \cite{cms98}.    Suppose  that $u$ is the corresponding interpolation function obtained by this algorithm.
 If $u$ is $C^2$ (even twice differentiable) at $x$,   the propagation here requires
  that \begin{align}\label{ia1}
u(x)=\frac 12\left[u(x+h Du(x))
-u(x-h Du(x)
)\right]+o(h^2).
\end{align}
By Taylor expansion, letting
$h\to 0$ one gets
\begin{align}\label{inter}\Delta_\fz u(x):=D^2 u(x)Du(x)\cdot Du(x)=0.\end{align}

  This enables us to consider the problem:
  \begin{align}\label{eicint}
\mbox{Find a function $u:\overline\Omega\to\rr$ to solve  \eqref{inter} in $\Omega$ with Dirichlet boundary condition \eqref{bv}.}
\end{align}
Note the above algorithm does not yield a $C^2$ interpolation function
necessarily, even  which
  is generally desirable in numerical analysis.
  One may ask whether  interpolation functions  or  solutions to \eqref{eicint} are  twice differentiable  almost everywhere so that \eqref{ia1} and \eqref{inter} hold almost everywhere.

\begin{center}
\begin{tikzpicture}
\draw plot[smooth] coordinates
{(-2,0.2) (-0.5, 0.5) (1, 0.2) (2.5, 0.5)}; 
\draw[smooth, dotted, thick, color=gray]  plot[smooth] coordinates{(-2,-0.8) (-0.5, -0.5) (1, -0.8) (2.5, -0.5)};\draw (-0.6, -0.5) node[below]{$x$};
\draw (-0.5, -0.5)--(-0.5, 0.5);\draw (-0.5, 0.5) node[above]{$x+hDu(x)$};

\draw (-0.5, -0.5)--(-0.4, -1.5);\draw (-0.4, -1.6) node[below]{$x-hDu(x)$};
\draw plot[smooth] coordinates{(-2,-1.8) (-0.5, -1.5) (1, -1.8) (2.5, -1.5)}; 
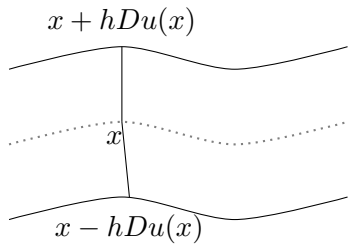
\draw (0,-2.4) node [below] {${\rm FIG \ 1: Interpolation\ by\ propagation}$};
\end{tikzpicture}
\end{center}

The equation  $\Delta_\fz u=0$ given in \eqref{inter} is exactly  the
$\fz$-Laplacian equation as given in Section \ref{Section1.1}, and its viscosity solutions are called
$\fz$-harmonic functions.  The equation \eqref{eic}  is written as the the normalized $\fz$-Laplacian equation  $\Delta^N_\fz u  =0$, where
 $$  \Delta^N_\fz u := \frac{\Delta_\fz u}{|Du|^2}=D|Du| \cdot \frac{Du}{|Du|}    \quad \mbox{when $u\in C^2$ and $Du\ne0$}.$$
  By Peres-Schramm-Sheffield-Wilson \cite{pssw09},  $\fz$-harmonic functions  are exactly viscosity solutions to the equation $\Delta^N_\fz u=0$.  For more background of $\fz$-harmonic functions we refer to the beginning of Section \ref{Section1}.
Therefore, the problem \eqref{eicpb} and the problem \eqref{eicint} lead to the Dirichlet problem
$$
 \mbox{ $\bdz_\fz u=0$ \mbox{in $\Omega$};
 $u=1$ on $S_1=\partial\Omega_1$ and $ u =0$ on $S_0=\partial\Omega_0$.}
$$
By Jensen \cite{j93},  the $\fz$-harmonic potential in $ \Omega$ is the unique viscosity solution to such
Dirichlet problem, and therefore 
 provides a viscosity solution to the problem \eqref{eicint}, and   also  a `` weak'' solution to \eqref{eicpb}.
 The regularity required by reconstruction functions in the problem \eqref{eicpb} and
 by interpolation functions in the problem \eqref{eicint}  lead to the study of the regularity of    $\fz$-harmonic potential, and in particular, Questions 1.1-1.3.
   Our finding in Section \ref{Section1.1} gives a limited insightful understanding
    of theses questions.

Since  Jensen \cite{j93} identified  $ \fz$-harmonic functions   with absolute minimizers for   $L^\fz$-functional
$F(v,\Omega)=\||Dv|^2\|_{L^\fz(\Omega)}$.
The {\it  equal importance criteria} somehow ask that the reconstruction
functions is absolute minimizer for $L^\fz$-functional
$F(v,\Omega)$ by  Aronsson in 1960's \cite{ar65, ar66,ar67}.
 Meanwhile, in the shape metamorphism and surface reconstruction,
  it is also quite often to  get the reconstruction  via  the constraint   minimizing the $L^p$-functional $$E_p(v,\Omega)=\frac1p\int_\Omega
  |Dv|^p\,dx,$$
  where $1<p<\fz$; see  Cong-Esser-Parvin-Bebis \cite{cepb04}.
The corresponding Euler-Lagrange equation reads as $\Delta_pv=0$ in $\Omega$,
  where the $p$-Laplacian
  $$ \Delta_pv={\rm div}(|Dv|^{p-2}Dv).$$
This leads to the  Dirichlet problem
 \begin{align}\label{acp}\Delta_pv=0 \mbox{ in $\Omega$}; \ v =0 \ {\rm on}\  \partial\Omega_0, \
v=1 \ {\rm on}\   \partial \Omega_1. \end{align}
This problem admits a unique weak solution $u_p\in W^{1,p}(\Omega)\cap C^{0}(\overline\Omega)$.
 This also arises from the an axiomatic approach to image interpolation,
 in Caselles-Morel-Sbert\cite{cms98}.
Such $u_p$ was already known to have some better regularity, say $C^{1,\gz}$-regularity and also some higher order regularity but all of them are not uniformly in  all $p\in (2,\fz)$.

It is valuable to comprehend  both the similarities and distinctions
among reconstruction functions obtained by different approaches.
Observe that, letting $p\to\fz$, the energy functional $E_p(v,\Omega)$ goes to the $L^\fz$-functional
$F(v,\Omega)$;
 the normalized $p$-Laplacian  $\frac1p|Dv|^{2-p}\Delta_pv$ converges to  the normalized $\fz$-Laplacian $\Delta_\fz^N v$ formally;
  the $p$-harmonic potentials  $u_p$  converges to
 the $\fz$-harmonic potentials  $u$  in $C^{0,\gz}(\Omega)$ and weakly in $W^{1,q}(\Omega)$.
 However, to consider the regularity and stability required by numeric analysis,
some higher order  approximation are expected. Our answers to Question 1.1A-1.3A
partially solve these problems.

\section{Basic properties of $\fz$-harmonic functions}\label{Section2}
In this section we recall the definition of $\fz$-harmonic functions and also several useful properties.  Let $U$ be an arbitrary bounded domain of $\rn$.

\begin{defn}
A function $v\in C^0(U)$ is  a viscosity subsolution to
\begin{align}\label{inf-eq}
-\bdz_\fz w=0\quad {\rm in}\quad U
\end{align}
 if 
at any $x_0\in U$, for any $\phi\in C^2(U)$ satisfying
$$v(x_0)-\phi(x_0)\ge v(x)-\phi(x)\quad \forall x\in B(x_0,r)\Subset U\ \mbox{for some $r>0$},$$
one has
$-\bdz_\fz \phi(x_0)\le 0.$

A function $v\in C^0(U)$ is a viscosity supersolution to
\eqref{inf-eq} if
$-v$ is a viscosity supersolution of \eqref{inf-eq}.

A function $v\in C^0(U)$ is
a viscosity solution to \eqref{inf-eq} if it is both a viscosity subsolution and a viscosity
supersolution to \eqref{inf-eq}.

Viscosity solutions to \eqref{inf-eq} are called $\fz$-harmonic functions in $U$.
\end{defn}

The following property is well-known; see for example  Crandall \cite{c08}.

\begin{thm} If $v\in C^{0,1}(\overline U) $ is $\fz$-harmonic in $U$,
then $\lip(v,U)=\lip(v,\partial U)$.
\end{thm}

  Jensen \cite{j93}  established a comparison principle.
 \begin{thm} If $v_1 \in C^{0,1}(\overline U) $ is viscosity subsolution to \eqref{inf-eq} in $U$ and $v_2 \in C^{0,1}(\overline U) $ is viscosity supersolution to \eqref{inf-eq} in $U$,
   then  $$\max_{\overline U} [v_1-v_2]=\max_{\partial U}[v_1-v_2].$$
\end{thm}

The comparison  with cones was introduced by
Crandall-Evans-Gariepy
\cite{ceg01}.
\begin{defn}
A function $v\in C^0(U)$ enjoys comparison with  cones
from above in $U$ if for any $V\Subset U$, $x_0\in U$, and
$b\ge 0$, it holds that
 $$\mbox{$v(x)\le v(x_0)+b|x-x_0|\quad \forall x\in V  $
 whenever
 $v(x)\le v(x_0)+b|x-x_0|\quad \forall x\in \partial (V\backslash\{x_0\}) . $}
$$

A function $v\in C^0(U)$ enjoys comparison with cones from below in $U$ if
 $-v$ enjoys comparison with cones from
above in $U$.

 A function $v\in C^0(U)$
enjoys comparison with cones in $U$  if it enjoys both comparison with cones from below and from above.
\end{defn}

A crucial fact for $\fz$-harmonic functions is that they can be
characterised through comparison with cones; see Crandall-Evans-Gariepy
\cite{ceg01}.

\begin{thm}\label{vis=cc}
Let $v\in C^0(U)$.
Then $v$ is a subsolution to \eqref{inf-eq} in $U$ if and only if $v$ satisfies   comparison with cones from above in $U$.

Consequently,  $v$ is $\fz$-harmonic if and only if $v$ satisfies   comparison  with cones.
\end{thm}

For $x\in   U$ and $0<r<\dist(x,\partial U)$, define the slope functions
$$S^+_r(v,x):=\max_{y\in \partial(B(x,r)\cap U)\backslash\{x\}}
\frac{v(y)-v(x)}{r}$$
and
$$S^-_r(v,x):=\max_{y\in \partial(B(x,r)\cap U)\backslash\{x\}}
\frac{v(x)-v(y)}{r}.$$
Denote by $\lip\, v(x)$ the pointwise Lipshictz constant at $x$, that is,
$$\lip\, v(x):=\limsup_{r\to 0}\sup_{|x-y|\le r}\frac{|u(y)-u(x)|}{r}.$$
 Crandall-Evans-Gariepy
\cite{ceg01} obtained the monotonicity of
the slope functions.

\begin{lem}\label{slope}
 Let $v$ be upper semi-continuous in $U $  and enjoy comparison with cones from above.
Then for any $x\in U$,  $S^+_r(v,x)$ is nondecreasing in $r\in(0,\dist(x,\partial U))$. 
 Moreover,
 $$\lim_{r\to 0}S^+_r(v,x)
=\lip\, v(x)\quad \mbox{for all $x\in U$.}$$
\end{lem}

As a  consequence of \eqref{vis=cc},
the following  strong maximum principle can be found in
Crandall \cite[Section 4]{c08}.
\begin{cor}
Let  $v$ be an $\fz$-harmonic
function  in a domain $U$.
 If $v(x)=\max
 _{\overline{B(x,r)}}v$ for some ball   $B(x,r)\Subset U$, then $v$ is a constant on $B(x,r/2)$.

 Consequently,   if $v$ attains its maximum  at any interior point of  $U$,
  then it must be a constant.
\end{cor}

Note that $p$-harmonic functions satisfies a similar strong maximum  principle; see Lindqvist \cite{l19}.

Evans-Smart \cite{es11a,es11b} proved the
everywhere differentiability of
$\fz$-harmonic functions.

\begin{thm}\label{everydiff}
If $v$ is an $\fz$-harmonic function   in a domain $U$, then  $v$ is  everywhere differentiable in $U$. Consequently,
$\lip\, v(x)=|Dv(x)|$ for all $x\in U$.
\end{thm}

\section{Quasi-concavity of  $p$-harmonic and $\fz$-harmonic potentials }\label{Section3}

In this paper, unless other specify, we always assume that
 $\Omega$ is  a convex ring, that is, $\Omega=\Omega_0\setminus
 \overline \Omega_1$,
 $\Omega_0$ is bounded convex domain and $\overline \Omega_1$
 is a convex compact subset of $\Omega_0$.
 We always denote by  $u\in C^0(\overline \Omega)\cap W^{1,\fz}(\Omega)$  the $\fz$-harmonic potential in  $ \Omega$.
 Naturally, one extends  $u$ to the whole domain $\rn$ by setting
$$\mbox{$u=1$ in $\overline \Omega_1$ and $u=0$ in $\rn\setminus \Omega_0$.}$$
 For each $p\in(2,\fz)$, we always denote by  $u_p\in C^0(\overline \Omega)\cap W^{1,p}(\Omega)$  the $p$-harmonic potential in $ \Omega$, that is, the unique weak solution to \eqref{acp}.
Naturally, one extends  $u_p$ to the whole domain $\rn$ by setting
$$\mbox{$u_p=1$ in $\overline \Omega_1$ and $u_p=0$ in $\rn\setminus \Omega_0$}.$$

 Let us review some result of $u_p$ due to Lewis \cite[Theorem 1]{l77}.
A function $f\in C^0(\rn)$ is called quasi-concave in  $\rn$ if for each $t\in\rr$ the super level set  $\{x\in \rn: f(x)>t\}$   is convex whenever it is not empty set, or equivalently,
 $$f(\lz x+(1-\lz)y)\ge \min\{f(x),f(y)\}\quad \forall \lz\in[0,1],x,y\in \rn.$$
\begin{lem}Let $p\in[2,\fz)$.  It holds that $u_p\in C^\fz(\Omega)$ with $Du_p\neq 0$ in $\Omega$,
$u_p$ is quasi-concave in $\rn$, and
 $-\bdz u_p>0$ in $\Omega$.
\end{lem}

Consequently,  we have the following.

\begin{lem} \label{lewis}
Let $p\in[2,\fz)$. It holds that
 $u_p$ is a viscosity subsolution to \eqref{inf-eq} in $\Omega$,
 and hence $u_p$ enjoys the comparison with
cones from above in $\Omega$. Moreover, $u_p\le u$ in $\Omega$.
\end{lem}
\begin{proof}
Using
$-\bdz_p u_p=0$ and $\Delta u_p<0$  in $\Omega$ for all $p>2$, we have
$$-\bdz_\fz u_p=\frac{1}{p-2}\bdz u_p|Du_p|^2\le 0\quad{\rm in}\quad \Omega .$$
This then guarantees that $u_p$ enjoys the comparison with
cones from above.  By the comparison principle, we have $u_p\le u$ in $\Omega.$
\end{proof}

Observe that $0<u_p<1$ in $\Omega$ via the strong maximum principle, and therefore, $$\Omega_0=\{x\in\overline\Omega_0|u_p(x)>0\}.$$
 For each $t\in[0,1)$, we write
 the super level set $$\Omega_t^p:=\{x\in \Omega_0| u _p(x)>t\}$$
 and the level set
 $$S^p_t:=\{x\in \Omega_0 |u_p(x)=t\}.$$
Note that
 $\Omega^p_0=\Omega_0$.
As a consequence of  Lemma \ref{lewis}, one has the following.

  \begin{lem}  For each $t\in[0,1)$,
 $\Omega_t^p $ is convex, and  its  boundary  $\partial\Omega_t^p=S^p_t$.
 If $t\in(0,1)$, then $S^p_t\subset\Omega$.
  \end{lem}

The following  was well-known; see \cite{bdbm}, \cite{lm95} and \cite[Lemma 3.2]{j96}.
\begin{lem}  As $p\to\fz$, one has
 $u_p\to u$ in $C^{0,\gz}(\overline \Omega)$ and
 $Du_p\to Du$ weakly in $L^q(\Omega)$ for any $\gz\in(0,1)$ and $q\ge 1$.
\end{lem}

 Moreover, it follows by the  strong maximum principle that $0<u(x)<1$ for $x\in\Omega$, and hence
$$\mbox{$\Omega_0=\{x\in\Omega_0|u(x)>0\}$ and $\partial  \Omega_0=\{x\in\overline \Omega_0|u(x)=0\} $.
}$$
 For each $t\in[0,1)$, we write
 the super level set $$\Omega_t:=\{x\in \Omega_0: u (x)>t\},$$
 and the level set
 $$S_t:=\{x\in \Omega_0 : u(x)=t\}.$$

Via the fact that  $u_p\to u$ in $C^0(\overline \Omega)$ as
 $p\to \fz$ and the strong maximum principle, we obtain the following property; see also
 \cite{ll19,ll21}.
 For reader's convenience, we give details of its proof.
  \begin{lem}
For each $t\in[0,1)$,
 $\Omega_t$ is convex, and moreover,
  its boundary   $\partial\Omega_t=S_t$.  If $t\in(0,1)$, then $S_t\subset\Omega_0$.
  \end{lem}

  \begin{proof}
   Let $t\in(0,1)$.
By the continuity of $u$, $\Omega_t$ is a non-empty  domain (connected open subset of $\Omega_0$).
 Since  $u_p\to u$  in $C^0(\overline \Omega_0)$, the convexity of
  the super level set of $u_p$ yields the convexity of $\Omega_t$.
 By $u\in C^0(\overline\Omega_0)$ and $0<u<1 $ in $\Omega$,
we know that
$$\mbox{ $\partial  \Omega_t\subset S_t$ and $  \Omega_t \subset  \overline \Omega_t  \subset \Omega_s$ whenever $0\le s<t<1$.}$$

Below we show that $S_t\subset\partial\Omega_t$
 by contradiction. Assume that $S_t\setminus \partial \Omega_t\ne\emptyset$.
 Write   $E_t:=\{x\in\Omega_0,u(x)\ge t\}$.
Noting that $\overline \Omega_t=\Omega_t\cup\partial\Omega_t$, we have
 $$ E_t\setminus\overline\Omega_t= S_t\setminus\partial \Omega_t\ne\emptyset.$$
Moreover, the convexity of $\Omega_t$ yields the convexity of
  $\overline \Omega_t$. Similarly, the convexity of $\Omega_t$ and continuity of $u$ further lead
  to  the convexity of
  $E_t$.
 Thus $(E_t)^\circ   \setminus
   \overline \Omega_t \ne\emptyset$.
  Indeed, take any $x_0\in E_t\setminus\overline\Omega_t$.
  By the convexity of $\overline\Omega_t$, there is unique $y_0\in\overline\Omega_t  $
  such that  $|x_0-y_0|={\rm dist}(x_0,\overline\Omega_t)>0$. Using a continuity of
  distance function ${\rm dist}(\cdot, \partial\Omega_0)$,
 we can find $\lz_t\in(0,1)$ such that  for any $ \lz\in(\lz_t,1)$,

 $${\rm dist}(\lz x_0+(1-\lz)z, \overline\Omega_t)\ge {\rm dist}(x_0,\overline\Omega_t)-| \lz x_0+(1-\lz)z-x_0|>0,
 \quad \forall z\in \overline \Omega_t.
 $$
 That is, the open set
 $$\{\lz x_0+(1-\lz)z| \lz\in(\lz_t,1),z\in\overline\Omega_t \} \subset E_t\setminus\overline\Omega_t.$$
Therefore  in  the domain $(E_t)^\circ\setminus \overline {\Omega_1}$, $u$ reaches its infimum$t$  at some interior point.
 However, since $u$ is $ \fz$-harmonic  in   $(E_t)^\circ\setminus \overline {\Omega_1}$,
 by the strong maximum principle and a covering argument,  we deduced that $u\equiv t$
 on $\Omega$ with $0<t<1$. This leads to a contradiction.
 \end{proof}

\section{A lower bound of  the length of  gradients}\label{Section4}

For an $\fz$-harmonic potential $u$ in a convex ring $ \Omega$, given  any $x\in\Omega$,
denote by $N_x$ the collection of all unit vector $\nu\in \rr^n$ so that
$$\nu\cdot (y-x)\le 0\quad\forall y\in\Omega_{u(x)}.$$
Since $\Omega_{u(x)}$ is convex,  one has $N_x\ne\emptyset$.
In other words,
 $N_x$ is the collection of  outer normal directions $\nu$ of all supporting hyperplane of  $\Omega_{u(x)}$ at $x$,  that is,  $\Omega_{u(x)}$ lies in the  side
$\{y\in\rn|(y-x)\cdot \nu<0\}$ of   the hyperplane  $P_\nu=\{y\in\rn|(y-x)\cdot \nu=0\}$.

We  have the following lower bound of the partial derivative of $u$ along
any direction in $  N_x$.
Note that, given any unit vector $\nu$, it follows from the convexity
 of $\Omega_0$ that
there
 is an unique   $r_{x,\nu}\in(0,\fz)$ such that
\begin{align}\label{rxnu} \mbox{$ x+r_{x,\nu}\nu\in \partial\Omega_0$, that is,   $u(x+r_{x,\nu}\nu)=0$. }
\end{align} If $ \nu\in N_x$, one then has 
\begin{align} \label{rxv}\dist(x,\partial\Omega_0)\le r_{x,\nu}\le\diam(\Omega_0).
\end{align}

 \begin{lem}\label{low-u}

   For any $x\in \Omega$,  it holds
\begin{align}\label{e3.x1}
 -Du(x)\cdot\nu \ge \frac{u(x)-u(x+r\nu_x)}{r}>0\quad\forall \nu\in N_x,\ r\in(0,r_{x,\nu}].
 \end{align}
 Moreover,
 \begin{align}\label{e3.x2}
 N_x=\left\{\nu_x:=-\frac{Du(x)}{|Du(x)|}\right\}.
 \end{align}
 Consequently,
\begin{align}\label{key-low3}
|Du(x)|\ge \frac{u(x)-u\left(x+r\nu_x\right)}{r}\quad\forall r\in(0,r_{x,\nu_x}].\end{align}
In particular,
\begin{align}\label{low-gru}
|Du(x)|\ge \frac{u(x)}{ r_{x,\nu_x}}.
\end{align}
 \end{lem}

\begin{proof}

{\it Proof of \eqref{e3.x1}.}
Assume that $x\in \Omega$. Given any $r\in(0,r_{x,\nu}]$, as  mentioned above we have $z=x+r\nu\in\overline \Omega_0\setminus \Omega_{u(x)}$.
Then it follows from the convexity of $\Omega_{u(x)}$ that $\overline {B(z,r)}\cap\overline{\Omega_{u(x)}}=\{x\}$ and
$B(z,r) \cap\Omega_{u(x)}=\emptyset$.
 Therefore
\begin{align*}
\max_{\overline{B(z,r)\cap \Omega_0}}u=u(x),
\end{align*}
 and then
$$ u(y)\le u(z)+u(x)-u(z)=u(z)+\frac{u(x)-u(z)}r|y-z|\quad\forall y\in \partial B(z,r)\cap \overline \Omega_0.$$
Since $u(x)\ge u(z)$, we get
 \begin{align*}
u(y)=0\le  u(z)+\frac{u(x)-u(z)}{r}|y-z|,\quad \forall y\in \partial \Omega_0.
\end{align*}
Noting $\partial (B(z,r)\cap \Omega_0)\subset
(\partial B(z,r)\setminus\overline  \Omega_0)\cup \partial\Omega_0$, and
by applying  comparison   with cones from above we obtain
 \begin{align*}
u(y)\le  u(z)+\frac{u(x)-u(z)}{r}|y-z|,\quad \forall y\in B(z,r)\cap \Omega_0.
\end{align*}
Recalling $z=x+r\nu$, then by setting $y=x+t\nu$
for $t\in(0,r) $ one has
\begin{align*}
u(x+t\nu)&\le  u(z)+ \frac{u(x)-u(z)}{r}|x+t\nu-z|   \\
&=u(z)+(r-t) \frac{u(x)-u(z)}{r}= u(x)-t \frac{u(x)-u(z)}{r},
\end{align*}
that is,
$$\frac{u(x)-u(x+t\nu)}t\ge \frac{u(x)-u(x+r\nu)}{r}.$$
By letting $t\to0$, it holds that
$$|Du(x)|\ge -Du(x)\cdot\nu\ge \frac{u(x)-u(x+r\nu)}{r}>0,$$
that is, \eqref{e3.x1} holds. 

{\it Proof of \eqref{e3.x2}.}
In order to prove $ \nu_x\in N_x$, it suffices to show that
$Du(x)\cdot (y-x)\ge 0$ for all $y\in\Omega_{u(x)}$.
For any $y\in\Omega_{u(x)}$, via the differentiability of $u$ in $\Omega$ we have that
$$ Du(x)\cdot \frac{y-x}{|y-x|}=
\lim_{t\to 0^+}\frac{u\left(x+t \frac{y-x}{|y-x|}\right)-u(x)}{t}. $$
It follows by the convexity of $\Omega_{u(x)}$ that
  $$x+t \frac{y-x}{|y-x|}=(1-\frac t{|y-x|})x+\frac{t}{|y-x|}y \in \Omega_{u(x)}\quad\forall 0<t<|y-x|.$$
  Thus $$u(x+t \frac{y-x}{|y-x|})>u(x) \quad \mbox{and then
  $Du(x)\cdot (y-x)\ge0$.}$$

Next, we show that   $N_x=\{\nu_x\}$ by contradiction.
Assume that there exists { another}
$\nu\in N_x$ with $ \nu\ne \nu_x$.
We claim that there is a unit vector $\eta$ such that
$\eta\cdot \nu_x<0$ while $\eta\cdot \nu>0$.
Indeed,  $\nu$ and $\nu_x$ span a 2-dimension plane $P$.
If  $\nu\cdot\nu_x<0$ we take $\eta=\nu$.
If $\nu\cdot\nu_x\ge 0$, denote by $\theta$ the angle between $\nu$ and $\nu_{x}$,
we let $\eta$ be the unit vector obtained by rotating
$\nu$ on the other side of $\nu_x$ with the angle $\frac\pi 2-\frac\theta 2$ and then
$\eta\cdot\nu=\cos(\frac\pi 2-\frac\theta 2)>0$.
The angle between $\eta $ and $ \nu_x$ is $\frac\pi 2+\frac\theta 2$ while
$\eta\cdot\nu_{x}=\cos(\frac\pi 2+\frac\theta 2)<0$. This gives the claim.

Via the claim above, we conclude that
$$0<|Du(x)|(-\nu_x)\cdot\eta=Du(x)\cdot \eta=\lim_{t\to0^+}\frac{u(x+t\eta)-u(x)}{t},$$
which implies that
$u(x+t\eta)>u(x)$ for all small $t>0$.
Therefore  for such $t>0$, one has $x+t\eta\in \Omega_{u(x)}$,
which together with the definition of $\nu$ yields that
$$t\eta\cdot\nu= (x+t\eta-x)\cdot\nu\le 0.$$
This is a contradiction to our choice of $\eta$ which yields  $\eta\cdot \nu>0$.
 \end{proof}

Given  any $x\in\Omega$,
denote by $N^p_x$ the collection of all unit vector $\nu$ so that
$$\nu\cdot (y-x)\le 0\quad\forall y\in\Omega^p_{u_p(x)}.$$
Since $\Omega_{u_p(x)}^p$ is convex,  then $N^p_x\ne\emptyset$.
Given any $\nu\in N^p_x$,  recall  $r_{x,\nu}$ as given in \eqref{rxnu}.
Then
 $ x+r _{x,\nu}\nu\in \partial\Omega_0$, that is,   $u_p(x+r _{x,\nu}\nu)=0$,
 and moreover,  $r _{x,\nu}$ also satisfies \eqref{rxv}.
As $u_p$   enjoys comparison with cones from above (see Section \ref{Section3}),  the following
 follows from an argument similar to the one  used to prove Lemma \ref{low-u}. We  omit the details of the proof.

 \begin{lem}\label{low-up}
   Let $p\in(2,\fz)$. For any $x\in \Omega$,  it holds
\begin{align}\label{e3.x1-p}
 -Du_p(x)\cdot\nu \ge \frac{u_p(x)-u_p(x+r\nu)}{r}>0\quad\forall \nu\in N^p_x\quad r\in(0,r_{x,\nu}].
 \end{align}
 Moreover, \begin{align}\label{e3.x2-p}
 N^p_x=\left\{\nu^p_x:=-\frac{Du_p(x)}{|Du_p(x)|}\right\}.
 \end{align}
 Consequently,
\begin{align}\label{key-low3-p}
|Du_p(x)|\ge \frac{u_p(x)-u_p\left(x+r\nu\right)}{r}\quad\forall \nu\in N_x^p\quad r\in(0,r _{x,\nu}].
\end{align}
In particular,
\begin{align}\label{low-grup}
|Du_p(x)|\ge \frac{u_p(x)}{ r_{x,\nu_{x}^p}}.
\end{align}
 \end{lem}
\begin{rem}\rm
Thanks to \eqref{low-gru} and \eqref{low-grup}, one concludes
 $$|Du(x)|\ge \frac{ u (x)}{{\rm diam}(\Omega_0)}
 \quad \mbox{and}\quad |Du_p(x)|\ge\frac{ u_p(x)}{{\rm diam}(\Omega_0)}\quad
\forall x\in \Omega.$$
This  lower bound of $|Du| $ and $|Du_p|$ was given by
Lindgren-Lindqvist \cite{ll21}.
\end{rem}

In this section, we finally give the following uniform lower bound and upper bound of $|Du_p|$.

\begin{lem}\label{un-bdup}
Given any ball $B(z,2r)\Subset \Omega$ we have
\begin{align}\label{uniforboundup}
\min_{\overline{B(z,r)}}|Du_p|\ge  \frac{1}{ {\rm diam}(\Omega_0)}\min\limits_{\overline{B(z,r)}}u_p
\quad {\rm and}\quad \max_{\overline{B(z,r)}}|Du_p|\le \frac1{r}.
\end{align}
Moreover, there exists a fixed constant $p_{z,r}>2$ such that
\begin{align}\label{uniforboundup-1}
\min_{\overline{B(z,r)}}|Du_p|\ge  \frac{1}{2 {\rm diam}(\Omega_0)}\min\limits_{\overline{B(z,r)}}u>0\quad \forall p> p_{z,r}.
\end{align}
  \end{lem}

  \begin{proof}
  Assume that $B(z,2r)\Subset\Omega$.
Applying the inequality \eqref{low-grup} in Lemma \ref{low-up} we obtain
\begin{align*}
\min_{\overline{B(z,r)}}|Du_p|\ge  \frac{1}{ {\rm diam}(\Omega_0)}\min\limits_{\overline{B(z,r)}}u_p \quad \forall p>2.
\end{align*}
Observe that $u_p\to u$ in $C^0(\overline \Omega)$ and $u\neq 0$ on $\overline {B(z,r)}$.
Via $\min\limits_{\overline{B(z,r)}}u>0$ we can find a fixed constant $p_{z,r}>2$ such that
$$\min\limits_{\overline{B(z,r)}}u_p\ge \frac 12\min\limits_{\overline{B(z,r)}}u\quad
\forall p> p_{z,r}.$$
This proves \eqref{uniforboundup-1}.

On the other hand, since $u_{p}$ enjoys the comparison with cones from above (see Section \ref{Section3}), by Lemma \ref{slope}
we have
$$|Du_{p}(x)|=\lip\, u_p(x)\le S^+_r(u_{p}, x)=\max_{|y-x|=r}\frac{u_{p }(y)-u_{p}(x)}{r }\quad \forall x\in B (z,r).$$
Thanks to $0\le u_p(x)\le 1$ and $B(z,2r)\Subset \Omega$, we obtain
$$\max_{\overline {B(z,r)}}|Du_{p}|\le \frac1{r}.$$
Hence we complete this proof.

 \end{proof}

\section{Proof of Theorem \ref{c1-re}}\label{Section5}
This section is devoted to proving Theorem \ref{c1-re}, which is split into
several lemmas.
The $C^1$ regularity of $u$ follows from Lemma \ref{uc1}. The locally uniform convergence of $Du_p$ relies on Lemma
\ref{upc1}.

Our first observation is continuity of $\nu_x=-\frac{Du(x)}{|Du(x)|}$ and
$|Du(x)|$ in $x\in\Omega$ coming from
Lemma \ref{low-u}.
\begin{lem}\label{uc1}
Suppose that $x_k\to x$ as $k\to \fz$. Then the following holds.

\begin{itemize}
\item[(i)] $\nu_{x_k}\to \nu_{x}$ as $k\to \fz$.

\item[(ii)] $|Du(x_k)|\to |Du(x)|$ as $k\to \fz$.
\end{itemize}
\end{lem}

\begin{proof}
{\it Proof of (i)}.
Let $\mu$ be one of  limit points of $\{\nu_{x_k}\}$. Up to a subsequence, we
may assume that $\nu_{x_k}\to\mu$. Below we show that $\mu=\nu_x$.
Given any $y\in\Omega_{u(x)}$, since $u(x_k)\to u(x)$, for all sufficiently large $k$
we have $y\in\Omega_{u(x_k)}$, and hence
$$\nu_{x_k}\cdot (y-x_k)\le 0.$$
Letting $k\to\fz$ one has
$$\mu\cdot (y-x)\le 0.$$
Then by \eqref{e3.x2} in Lemma \ref{low-u} we have $\mu=\nu_{x}$. This proves (i).

{\it Proof of (ii)}.
We already know that
$|Du(x)|\ge\limsup_{k\to\fz}|Du(x_k)|$ by the upper semicontinuous of $|Du|$; see \cite{ceg01}.
Thus it suffices to prove 
\begin{align}\label{liminf}\liminf_{k\to\fz}|Du(x_k)|\ge |Du(x)|.
\end{align}

Recall that \eqref{key-low3} in Lemma \ref{low-u} give us
$$|Du(x_k)|\ge \frac{u(x_k)-u\left(x_k+t\nu_{x_k} \right)}{t}
\quad \forall k\ge 1,\quad \forall 0<t<\frac14{\rm dist}(x_k, \partial\Omega_0).$$
Sending  $k\to \fz$, by $\nu_{x_k}\to\nu_x$ in the proof of (i) and $u\in C^0(\Omega_0)$ one gets
$$u(x_k)-u\left(x_k+t\nu_{x_k} \right)\to u(x)-u(x+t\nu)\quad \mbox{and}\quad
{\rm dist}(x_k, \partial\Omega_0)\to {\rm dist}(x, \partial\Omega_0),$$
and hence
\begin{align*}
\liminf_{k\to\fz}|Du(x_k)|\ge \frac{u(x)-u\left(x+t\nu_x\right)}{t}
\quad \forall 0<t<\frac14{\rm dist}(x, \partial\Omega_0).
\end{align*}
Observe that
\begin{align*}
u\left(x +t\nu_x\right)-u(x)=-t|Du(x )|+o(t)
\end{align*}
and hence
$$\frac{u(x)-u\left(x+t\nu_x\right)}{t}=|Du(x)|
+o(1).$$
 As $ t\to0$, it follows that  \eqref{liminf}. Hence we complete this proof.
\end{proof}

To obtain the locally uniform convergence of $Du_p$, we also need the following lemma,
which is exactly the same as Lemma \ref{uc1}.
\begin{lem}\label{upc1}
Suppose that $x_j\to x\in\Omega$ and $p_j\to\fz$ as $j\to\fz$. Then the following holds.
\begin{itemize}
\item[(i)]  $ \nu^{p_j}_{x_j}\to \nu_x$ as $j\to\fz$.

\item[(ii)] $|Du_{p_j}(x_j)|\to |Du(x)|$ as $j\to\fz$.
\end{itemize}
\end{lem}

\begin{proof}
{\it Proof of (i)}.
Thanks to Lemma \ref{un-bdup},
up to a subsequence, we may assume that
$ \nu^{p_j}_{x_j}$ converges to some unit vector $ \mu$.
It is enough to show that
$\mu=\nu_x$.
Due to Lemma \ref{low-u}, we only need to show that
$\mu\in N_x$, that is,
$$\mu\cdot (y-x)\le 0\quad\forall y\in\Omega_{u(x)}.$$

Given  any $y\in \Omega_{u(x)}$,   one has $u(y)>u(x)$.
Recall that $u_{p_{j }}\to u$ uniformly in $\overline \Omega $ and $x_{j }\to x\in \Omega$ as $j\to \fz$, we know that $u_{p_j}(x_j)\to u(x)$, $u_{p_j}(y)\to u(y)$ and $u_{p_j}(x_j)\to u(x)$ as $j\to\fz$.
For sufficiently large $j$, it then follows  $u_{p_{j }}(y)>u_{p_j}(x_j)$, that is,    $y\in\Omega^{p_j}_{u_j(x)}.$
By Lemma \ref{low-up}, one  gets
\begin{align*}
 \nu^{p_j}_{x_j}\cdot (y-x_j)\le 0.
\end{align*}
Letting $j\to\fz$,  we conclude that $\mu\cdot (y-x )\le 0$ as desired.

{\it Proof of (ii)}.
Since $x_j\to x\in\Omega$, via Lemma \ref{un-bdup}
 we know that $\{ |Du_{p_j}(x_j)|\}_j$ is bounded uniformly $p_j>2$.

We first show $\limsup_{j\to\fz}|Du_{p_j}(x_j)|\le |Du(x)|$. Indeed,
 since $x_j\to x$ as $ j\to\fz$,
for all sufficiently large $j$ so that $\frac 12{\rm dist}(x, \partial\Omega)\le {\rm dist}(x_j, \partial\Omega)$.
For such $j$,  given any $0< r< \frac 12{\rm dist}(x, \partial\Omega)$,
by Lemma \ref{slope} and Lemma \ref{lewis}, we have
$
 |Du_{p_j}(x_j)| \le  S^+_r(u_{p_j}, x_j) $.
Choose $y_{j,r}\in \partial B(x_j,r)$ so that
$$S^+_r(u_{p_j}, x_j)=\max_{|y-x_j|=r}\frac{u_{p _j}(y)-u_{p_j}(x_j)}{r }
=\frac{u_{p _j}(y_{j,r})-u_{p_j}(x_j)}{r }.$$
We then have
$$
 |Du_{p_j}(x_j)| \le
 \frac{u_{p _j}(y_{j,r})-u_{p_j}(x_j)}{r }.$$
Noting that $$r-|x_j-x|=|y_{j,r}-x_j| -|x_j-x|\le |y_{j,r}-x|\le |y_{j,r}-x_j| +|x_j-x |=r+|x_j-x|,$$
one has $\lim_{j\to\fz} |y_{j,r}-x|=r$, and hence,
  all limits of $y_{j,r}$ belong to $\partial B(x,r)$.
  Since  $u_{p_j}\to u$ uniformly in $\overline \Omega$,  we get
\begin{align*}
  \limsup_{j\to\fz}|Du_{p_j}(x_j)|
&=\limsup_{j\to\fz}\frac{u_{p _j}(y_{j,r})-u_{p_j}(x_j)}{r }\le\sup_{y\in \partial B(x,r)}
 \frac{u (y )-u (x)}{r }= S^+_r(u,x).
\end{align*}
Therefore, by Lemma \ref{slope} and Theorem \ref{everydiff}, one has
$$  \limsup_{j\to\fz}|Du_{p_j}(x_j)| \le \lim_{r\to0}S^+_r(u,x)=|Du(x)|.$$

 On the other hand,
by \eqref{key-low3-p} in Lemma \ref{low-up} we have
$$|Du_{p_j}(x_j)|\ge \frac{u_{p_j}(x_j)-u_{p_j}\left(x_j+t\nu_{x_j}^{p_j} \right)}{t}
\quad \forall j\ge 1,\quad \forall t<\frac 14{\rm dist}(x_j,\partial\Omega_0).$$
Since $u_{p_j}(x_j)\to u(x)$ and $\nu_{x_j}^{p_j}\to \nu_x$ due to (i) in this lemma,
letting $j\to \fz$ we get
$$\liminf_{j\to\fz}|Du_{p_j}(x_j)|\ge \frac{u(x)-u\left(x+t\nu_x\right)}{t}
,\quad \forall 0<t<\frac 14{\rm dist}(x,\partial\Omega_0) .$$
Using the differentiability of $u$ in $\Omega$, by sending $t\to 0$ we conclude
$$\liminf_{j\to\fz}|Du_{p_j}(x_j)|\ge -Du(x)\cdot\nu_x=|Du(x)|.$$
Hence we finish this proof.
\end{proof}

Let us now proof Theorem \ref{c1-re}.

\begin{proof}[Proof of Theorem \ref{c1-re}]
The $C^1$ regularity of $u$ follows from Lemma \ref{uc1}.
Also, the locally uniform convergence of $Du_p\to Du$ is a direct consequence of
Lemma \ref{upc1}.
\end{proof}

\section{Two fundamental inequalities under quasi-concavity}\label{Section6}

We establish a fundamental algebraic and geometric structural  inequalities involving $ \Delta_\fz v$ for smooth
 quasi-concave functions $v$
 via the geometry of the contour  surfaces.
\begin{lem}\label{id1}
Let $n\ge 2$ and let $v:\rn \to\rr$ be a quasi-concave function.
If  $v\in C^\fz(U)$ for some domain $U\subset \rn$, then
\begin{align}\label{e6.x1}2[| D^2v Dv|^2-\bdz v\bdz_\fz v]\ge |Dv|^2[| D^2v|^2-(\bdz v)^2] \quad{\rm in}\quad U.
\end{align}
\end{lem}
This further allows us to get the following fundamental divergence inequality
for quasi-concave functions.

\begin{lem}\label{id2}
Let $n\ge 2$ and let $v:\rn \to\rr$ be a quasi-concave function.
If  $v\in C^\fz(U)$ for some domain $U\subset \rn$, then
\begin{align*}
{\rm div}\left(|Dv|^{-2}(\bdz v Dv-D ^2v Dv)\right)\ge 0\quad{\rm in}\quad U\backslash \{Dv\neq 0\}.\end{align*}
\end{lem}

\begin{rem}\label{rem-n-2}
\rm
In dimension $n=2$,
for any function $v\in C^\fz(U)$ it was shown in \cite{kzz} that
$$2[| D^2v Dv|^2-\bdz v\bdz_\fz v]= |Dv|^2[| D^2v|^2-(\bdz v)^2] \quad{\rm in}\quad U
$$
and
$${\rm div}\left(|Dv|^{-2}(\bdz v Dv-D ^2v Dv)\right)=0 \quad{\rm in}\quad U\backslash \{Dv\neq 0\}.$$
So the main ingredient for Lemma \ref{id1} and Lemma \ref{id2} is in dimension $n\ge3$.
\end{rem}

We need the following lemma to prove Lemma \ref{id1}. We denote
$e_n=(0,...,0,1)$.
\begin{lem}\label{det-2}
Let $n\ge 3$ and let $v:\rn \to\rr$ be a quasi-concave function.
Assume that $v\in C^2(B(0,r))$, $v(0)=0$ and $Dv(0)=e_n$.
Then
\begin{align}\label{id1-2}
 (w_{x_ix_j})^2-w_{x_ix_i}w_{x_jx_j}\le 0 \quad \forall 1\le i<j\le n-1.
\end{align}
\end{lem}
The inequality \eqref{id1-2} essentially says that the sectional curvature is non-positive at $0$.
This is well-known in differential geometry.
For reader's convenience, we give the details via an  analysis argument.

\begin{proof}
First, note that the quasi-concavity of $v$ guarantees
\begin{align}\label{id1-1a}
\mbox{
 $v(y)\le 0$ for all $y\in B(0,s) $ with $y_n=0$ }  \end{align}
 for some $0<s<r$.
Indeed, since $v\in C^2(B(0,r))$ and $Dv(0)=e_n$, by the implicit theorem we know that
the equation $v(y)=0$ determines a $(n-1)$-dimension surface $S$ in a ball  $B(0,s)$
for some $0<s<r$.  The boundary of
$ \{y\in B(0,s)|v(y)>0\}$ is $(S\cap B(0,s))\cup (\partial B(0,s)\cap \{v>0\})$.
Since $Dv(0)=e_n$, the tangential plane of $S$ at $0$ is given by the hyperplane
$P=\{y\in\rn|y_n=0 \}$.  The convexity of
$\{y\in\rn |v(y)>0\}$  implies that the convexity of $\{y\in B(0,s)|v(y)>0\}$.
Thus
$\{y\in B(x,s)|v(y)>0\}$ lies in above $P$, that is,
if $v(y)>0$, then $ y_n> 0 $.
 Therefore, \eqref{id1-1a} holds as desired.

Applying  \eqref{id1-1a},  we claim that
\begin{align}\label{id1-1}
 D^2v(0)\xi \cdot \xi \le 0\quad \forall \xi \in \rn \quad
{\rm with } \quad \xi_n=0,
\end{align}
that is, the second fundamental form of $S$ at point $0$ is non-positive definite.
Indeed, given any $\xi\in\partial B(0,1)$,
by Taylor's expansion  and $Dv(0)=e_n$ one has
$$v( t\xi)=v(0)+t\xi_n+\frac{t^2}2 D^2v(0)\xi \cdot \xi  +o(t^2).$$
If $\xi_n=0$,
by \eqref{id1-1a} one then has $v( t\xi)\le v(0)=0$   and hence
 $$D^2v(0)\xi \cdot \xi =  \lim_{t\to0}\frac2{t^2} [v( t\xi)-v(0)]\le 0.$$
This gives the claim \eqref{id1-1}.

Next,  at point $0$ we apply \eqref{id1-1}   to get
\begin{align*}
 (w_{x_ix_j})^2-w_{x_ix_i}w_{x_jx_j}\le 0 \quad \forall 1\le i<j\le n-1.
\end{align*}
Indeed, given any pair $(i,j)$ with $1\le i<j\le n-1$, write
 \begin{equation*}
	Q:=\begin{bmatrix}
	 w_{x_ix_i} &\  w_{x_ix_j}  \\
	w_{x_jx_i} &\  w_{x_jx_j}
	 \end{bmatrix}.
\end{equation*}
Applying  \eqref{id1-1} with $\xi=a e_i+be_j$ for all
$a,b\in \rr$ we obtain
$$(a,b) Q(a,b)^T=a^2 w_{x_ix_i}+b^2 w_{x_jx_j}+2ab w_{x_ix_j}=\xi^TD^2w(0)\xi\le 0.$$
Clearly, this quadratic form is negative semi-definite,
and hence 
its eigenvalues of $Q$ are non-positive. Thus
$$\det Q=w_{x_ix_i}w_{x_jx_j}-(w_{x_ix_j})^2\ge 0,$$
which gives \eqref{id1-2}.
\end{proof}

We are ready to prove Lemma \ref{id1}.
\begin{proof}[Proof of Lemma \ref{id1}]
By Remark \ref{rem-n-2}, we only need to consider  the dimension $n\ge 3$.
Assume that  $v:\rn \to\rr$ is a quasi-concave function and $v\in C^\fz(U)$.
Fix any $x\in U$. Up to considering $v(z+x)-v(x)$ with $z+x\in U$ we may assume that
$x=0\in U$ and $v(0)=0$.
If $Dv(0)=0$, then
\eqref{e6.x1} holds trivially.
So we assume $Dv(0)\ne 0$ below.
Up to considering $v/|Dv(0)|$, we may assume that $|Dv(0)|=1$.
Below we consider the following two cases: $Dv(0)=e_n$, and $Dv(0)\ne e_n$.

{\it Case  $Dv(0)=e_n$.}  At the point $0$, by
$D v=e_n$, a direct  calculation gives
$$\bdz_\fz v=v_{x_nx_n},\quad | D^2v Dv|^2=\sum_{1\le i\le n} (v_{x_ix_n})^2.$$
 We then write \eqref{e6.x1} as
$$2 \sum_{1\le i\le n} (v_{x_ix_n})^2-2\bdz v v_{x_nx_n}\ge  | D^2v|^2-(\bdz v)^2,$$
 that is,
\begin{align}\label{e6.x3}| D^2v|^2\le (\bdz v)^2-2\bdz v v_{x_nx_n}+ 2\sum_{1\le i\le n} (v_{x_ix_n})^2.\end{align}

To prove \eqref{e6.x3}, we write
\begin{align*}
| D^2 v|^2&= \sum_{i=1}^{n }\left(  v_{x_ix_i}\right)^2+ 2\sum_{1\le i<j\le n}\left( v_{x_ix_j}\right)^2\\
&= \sum_{i=1}^{n-1}\left(  v_{x_ix_i}\right)^2+
( v_{x_nx_n})^2+  2\sum_{1\le i<j\le n-1}\left( v_{x_ix_j}\right)^2+2\sum_{1\le i\le n-1} ( v_{x_ix_n})^2.
\end{align*}
Note that
\begin{align*}
\sum_{i=1}^{n-1}\left(  v_{x_ix_i}\right)^2&=
\left(\sum_{i=1}^{n-1} v_{x_ix_i}\right)^2-2\sum_{1\le i<j\le n-1} v_{x_ix_i} v_{x_jx_j}\\
&= (\Delta  v-  v_{x_nx_n})^2-2\sum_{1\le i<j\le n-1} v_{x_ix_i} v_{x_jx_j}\\
&=(\Delta v)^2-2\Delta v  v_{x_nx_n}+  (v_{x_nx_n})^2-2\sum_{1\le i<j\le n-1} v_{x_ix_i} v_{x_jx_j}
\end{align*}
and
$$2\sum_{1\le i\le n-1} ( v_{x_ix_n})^2 =
2\sum_{1\le i\le n} ( v_{x_ix_n})^2  -2( v_{x_nx_n})^2 .$$
 We therefore obtain
\begin{align*}
 | D^2 v|^2
&=(\Delta v)^2-2\Delta v  v_{x_nx_n}+ 2( v_{x_nx_n})^2\nonumber\\
&\quad\quad +2\sum_{1\le i<j\le n-1}[( v_{x_ix_j})^2- v_{x_ix_i} v_{x_jx_j}]
+ 2\sum_{1\le i\le n} ( v_{x_ix_n})^2-2( v_{x_nx_n})^2 \\
&= (\Delta v)^2-2\Delta v  v_{x_nx_n}+2\sum_{1\le i\le n} ( v_{x_ix_n})^2+2\sum_{1\le i<j\le n-1}[( v_{x_ix_j})^2- v_{x_ix_i} v_{x_jx_j}].
\end{align*}
Since Lemma \ref{det-2} tells us that
$$
\sum_{1\le i<j\le n-1}[( v_{x_ix_j})^2- v_{x_ix_i} v_{x_jx_j}]\le 0,$$
we get \eqref{e6.x3} as desired.

{\it Case $Dv(0)\ne e_n$}.
Denote by  $O$  the orthogonal matrix  so that
$Dv(0)=Oe_n$.
 Define
 $$w(y):= v(Oy)\quad \forall y\in U.$$
  At point $0$, we have $Dw=O^TDv=e_n$. Applying  \eqref{e6.x3} to $w$, we get
$$2[| D^2wDw|^2-\bdz w\bdz_\fz w]\ge  [| D^2w|^2-(\bdz w)^2] .$$
This would imply \eqref{e6.x1}    for $v$ at $0$ once we have
$$\mbox{$ |D^2wDw|^2=| D^2v   Dv |^2$,  $|D^2w|^2=|D^2v|^2$, $\Delta w=\Delta v$ and
$\Delta_\fz w=\Delta v_\fz$ at 0.}$$
 We check this as below.
 At point $0$,
 we have
$$\mbox{$  D^2w = O^TD^2v O$  and $D^2w Dw= O^T   D^2v(OO^T)Dv=O^T   D^2vDv .$}$$
Thus $ |D^2w|^2=|D^2v|^2$ and $| D^2w Dw |^2 =
| D^2v   Dv |^2.$   By $OO^T=I_n$, we also have
\begin{align*}
\Delta_\fz w&=
(Dw )^TD^2w Dw  =  (Dv)^T (O O^T)  D^2v Dv
 = (Dv )^T D^2v  Dv = \bdz_\fz v .
\end{align*}
Moreover, the cyclic property of trace implies that
\begin{align*}
\Delta  w&= {\rm tr}(D^2w) = {\rm tr} (O^TD^2v O)={\rm tr} (OO^TD^2v)=
{\rm tr}  (D^2v)=\Delta v.
\end{align*}
The proof is complete.
\end{proof}

Finally we apply Lemma \ref{id1} to prove Lemma \ref{id2}.
\begin{proof}[Proof of Lemma \ref{id2}]
At any point where $Dv\neq 0$, a direct calculation gives
\begin{align*}
{\rm div}(|Dv|^{-2}\bdz v Dv)=-2|Dv|^{-4}\bdz_\fz v \bdz v
+|Dv|^{-2}(\bdz v)^2+|Dv|^{-2}D\bdz v\cdot Dv
\end{align*}
and
\begin{align*}
{\rm div}(|Dv|^{-2} D^2v Dv)=-2|Dv|^{-4}| D^2v Dv|^2+ |Dv|^{-2}D \bdz v\cdot Dv
+|Dv|^{-2}| D^2v|^2.
\end{align*}
Combining them  and  using Lemma \ref{id1} 
 one then obtains
\begin{align*}
&{\rm div}\left(|Dv|^{-2}\bdz v Dv-|Dv|^{-2}D ^2v Dv\right)\\
&=2|Dv|^{-4}[| D^2vDv|^2-\bdz v\bdz_\fz v]
+|Dv|^{-2}[(\bdz v)^2-| D^2v|^2]\ge0
\end{align*}
as desired. 
The proof is complete.  
\end{proof}

\section{Proof of Theorem \ref{sob-re}  }\label{Section7}
In this section we  prove Theorem \ref{sob-re}.
Firstly,  via Lemma \ref{id2} we established the following
 upper bound for $ D|Du_p|^{\az}$ with $\alpha< 0$, which is uniform in all $p\in[2,\fz)$.
Recall that $u_p$ is the $p$-harmonic potential and
 $u$ is the $\fz$-harmonic potential in the convex ring $\Omega$.

\begin{lem}\label{sob-es}
 Given any $ \alpha<0$, for any $p\in(2,\fz)$ we have
 \begin{align}\label{sob-eseq}
\int_{B(z,r/2)}| D|Du_p|^{\az}|^2\,dx\le C(1+|\az|^2) r^{-2}\int_{B(z,r)}|Du_p|^{2\az}\,dx
\quad\mbox{whenever} \quad B(z,r)\Subset\Omega,
\end{align}
where $C$ is a universal constant.
\end{lem}
\begin{proof}
Since $Du_p\neq 0$ on $\Omega $, applying Lemma \ref{id2} to $u_p$ one has
\begin{align*}
{\rm div}\left(|Du_p|^{-2}[\bdz u_p Du_p- D^2u_pDu_p]\right)\ge 0\quad{\rm in}\quad B(z,r)\Subset \Omega.
\end{align*}
Multiplying both sides by a test function $\xi^2|Du_p|^{2\az}$ with $\az<0$ and
$\xi\in C^\fz_c(B(z,r))$, one gets
\begin{align*}
\int_{B(z,r)}
{\rm div}\left(|Du_p|^{-2}[\bdz u_p Du_p- D^2u_pDu_p]\right)\xi^2|Du_p|^{2\az}\,dx\ge 0.
\end{align*}
Via integration by parts, it becomes
\begin{align}\label{x7.0}
&-2\az\int_{B(z,r)}|Du_p|^{2\az-4}| D^2u_pDu_p|^2\xi^2\,dx
+2\az\int_{B(z,r)}|Du_p|^{2\az-4}\bdz u_p \bdz_\fz u_p\xi^2\,dx\nonumber\\
 &\le -2\int_{B(z,r)}|Du_p|^{2\az-2}\bdz u_p Du_p\cdot D \xi \xi\,dx
+2\int_{B(z,r)}|Du_p|^{2\az-2} D^2u_p Du_p\cdot D \xi \xi\,dx.
\end{align}

For the second term in the left-hand side of \eqref{x7.0}, using  $-\bdz^N_\fz u_p=\frac{1}{p-2}\bdz u_p$ in $\Omega $, by $\alpha<0$ and $p>2$
we have
$$2\az\int_{B(z,r)}|Du_p|^{2\az-4}\bdz u_p \bdz_\fz u_p\xi^2\,dx= -\frac{2\az}{p-2}\int_{B(z,r)}|Du_p|^{2\az-2}(\bdz u_p)^2\xi^2\,dx\ge0.$$

For the first term of right-hand side of \eqref{x7.0}, by integration by parts again
and using $\az<0$,
we obtain
\begin{align*}
&-2\int_{B(z,r)}|Du_p|^{2\az-2}\bdz u_p Du_p\cdot D\xi \xi\,dx\nonumber\\
&= 2\int_{B(z,r)}|Du_p|^{2\az-2} D^2u_p Du_p\cdot D \xi \xi\,dx
+4(\az-1)\int_{B(z,r)}|Du_p|^{2\az-4}\bdz_\fz u_p Du_p\cdot D  \xi \xi\,dx\nonumber\\
&\quad+2\int_{B(z,r)}|Du_p|^{2\az-2} D^2\xi Du_p\cdot Du_p \xi\,dx
+2\int_{B(z,r)}|Du_p|^{2\az-2}(Du_p\cdot D \xi)^2\,dx\nonumber\\
&\le -\frac \az4\int_{B(z,r)}|Du_p|^{2\az-4}| D^2u_p Du_p|^2\xi^2\,dx
+C[1+|\alpha|+\frac1{|\alpha|}]\int_{B(z,r)}|Du_p|^{2\az}(|D\xi|^2+| D^2\xi||\xi|)\,dx,
\end{align*}
where we also used the Young's inequality and Cauchy-Schwartz' inequality  in the last inequality.

For the second  term of right-hand side of \eqref{x7.0}, it follows by  Young's inequality
that
\begin{align*}
&2\int_{B(z,r)}|Du_p|^{2\az-2} D^2u_p Du_p\cdot D \xi \xi\,dx\\
 &\le -\frac \az4\int_{B(z,r)}|Du_p|^{2\az-4}| D^2u_p Du_p|^2\xi^2\,dx+
C\frac1{|\alpha|}\int_{B(z,r)}|Du_p|^{2\az }  |D \xi|\,dx.
\end{align*}

Combining above we have
$$-\frac\alpha 2\int_{B(z,r)}|Du_p|^{2\az-4}| D^2u_p Du_p|^2\xi^2\,dx\le C[1+|\alpha|+\frac1{|\alpha|}]\int_{B(z,r)}|Du_p|^{2\az}(|D \xi|^2+| D^2\xi||\xi|)\,dx.$$
Since  $Du_p\neq 0$ on $B(z,r)$,
 rewriting
$$|Du_p|^{2\az-4}| D^2u_p Du_p|^2=\frac1{\alpha^2}| D|Du_p|^{\az}|^2,$$
we deduce
\begin{align*}
\int_{B(z,r)}| D|Du_p|^{\az}|^2\xi^2\,dx\le C(1+|\az|^2)\int_{B(z,r)}|Du_p|^{2\az}(|D \xi|^2+| D^2\xi||\xi|)\,dx.
\end{align*}
Choosing a suitable cut-off function $\xi\in C^\fz_c(B(z,r))$
we obtain
\begin{align*}
\int_{B(z,r/2)}| D|Du_p|^{\az}|^2\,dx\le C(1+|\az|^2)r^{-2}\int_{B(z,r)}|Du_p|^{2\az}\,dx
\end{align*}
as desired.
\end{proof}

Consequently, we have the following upper bound for $ D|Du_p|^{\az}$ with $\alpha\ge 0$, which is uniform in $p>2$.  Recall that when $\alpha=0$,  $  f^{\az}$
is understood as $\ln f$.
 \begin{lem}\label{sob-es-1}
Given any $\az\ge 0$,  for any $p\in(2,\fz)$ we have
\begin{align*}
\int_{B(z,r/2)}| D|Du_p|^{\az}|^2\,dx\le C \az^2 r^{-2}\max_{B(z,r/2)}|Du_p|^{2\az+2}\int_{B(z,r)}|Du_p|^{-2 }\,dx \quad \mbox{whenever $ B(z,r)\Subset\Omega$,}
\end{align*}
where $C$ is a universal constant.
\end{lem}

\begin{proof}
Let $B(z,r)\Subset \Omega $ and $\az\ge 0$.
Since $Du_p\neq 0$ in $\Omega$,  one has
 $$ | D|Du_p|^{\az}| =| D(|Du_p|^{-1})^{-\az}| = \alpha
(|Du_p|^{-1})^{-\az-1} | D|Du_p|^{-1}|=
 \alpha  |Du_p|^{1+\az}| D|Du_p|^{-1}|.$$
Then we conclude from \eqref{sob-eseq} in Lemma \ref{sob-es} that
\begin{align*}
\int_{B(z,r/2)}| D|Du_p|^{\az}|^2\,dx&\le \az^2\max_{B(z,r/2)}|Du_p|^{2\az+2}
\int_{B(z,r/2)}| D|Du_p|^{-1}|^2\,dx\\
&\le C\az^2 r^{-2}\max_{B(z,r/2)}|Du_p|^{2\az+2}\int_{B(z,r)}|Du_p|^{-2 }\,dx.
\end{align*}

Since
$$
|D \ln |Du_p||=  |D \ln |Du_p|^{-1}|=
|Du||D  |Du_p|^{-1}|,$$
it follows from \eqref{sob-eseq} in Lemma \ref{sob-es} again that
\begin{align*}
\int_{B(z,r/2)}|D \ln |Du_p||^2\,dx&\le \max_{B(z,r)}|Du_p|^2
\int_{B(z,r/2)}| D|Du_p|^{-1}|^2\,dx\\
&\le C r^{-2}\max_{B(z,r/2)}|Du_p|^{2 }\int_{B(z,r)}|Du_p|^{-2 }\,dx.
\end{align*}

\end{proof}

Thanks to the local uniform bound of $|Du_p|$ in Lemma \ref{un-bdup}, sending $p\to\fz$, we get the following.
\begin{lem}\label{sob-re-x}
It holds that
 $ |Du|^\alpha\in W^{1,2}_\loc(\Omega)$ for any $\alpha\in\rr$
   and
    $|Du_p|^\alpha\to |Du|^\alpha$ weakly in $ W^{1,2}_\loc(\Omega)$ for any $\alpha\in\rr $ as $p\to\fz$.
 Moreover, we have the following quantitative upper bound.
If $ \alpha<0$,  then
 \begin{align*}
\int_{B(z,r/2)}| D|Du |^{\az}|^2\,dx\le C(1+|\az|^2) r^{-2}\int_{B(z,r)}|Du |^{2\az}\,dx
\quad\mbox{whenever} \quad B(z,r)\Subset\Omega.
\end{align*}
If $\az\ge 0$,  then
\begin{align*}
\int_{B(z,r/2)}| D|Du|^{\az}|^2\,dx\le C \az^2 r^{-2}\max_{B(z,r/2)}|Du|^{2\az+2}\int_{B(z,r)}|Du|^{-2 }\,dx
\quad\mbox{whenever} \quad B(z,r)\Subset\Omega.
\end{align*}
Here $C$ are  universal constants.
\end{lem}

\begin{proof}
We only consider the case $\alpha<0$. The case $\alpha\ge 0$ is similar; we omit the details.

For any $B(z,r)\Subset\Omega$  we recall  from   Lemma \ref{sob-es} that
\begin{align*}
\int_{B(z,r/2)}| D|Du_p|^{\az}|^2\,dx&\le C(1+\az^2)r^{-2}\int_{B(z,r)}|Du_p|^{2\az},
\end{align*}
where $C$ is a universal constant.
By \eqref{uniforboundup} and \eqref{uniforboundup-1} in Lemma \ref{un-bdup},
 we have $|Du_p|^{\az}\in W^{1,2}(B (z,r/2))$ uniformly $p>2$.
Using the weak compactness
of Sobolev space  $W^{1,2}(B (z,r/2))$(see \cite[Section 5.7]{e98}),
there exists a function $g_{\az}\in W^{1,2}(B(z,r/2)) $
such that
\begin{align*}
 D|Du_p|^{\az}\to Dg_{\az}\ {\rm weakly\ in}\  L^2(B(z,r))\
{\rm and}\ |Du_p|^{\az}\to g_{\az}\ {\rm in}\  L^2(B(z,r)) \mbox{ as $p\to \fz$.}
\end{align*}
Since Theorem  \ref{c1-re} gives $Du_p\to Du$ in $C^1(\Omega)$,
via Lemma \ref{un-bdup} we know that
\begin{align*}
\mbox{
$|Du_p|^\alpha\to   |Du|^\alpha$ in $ C^0(\Omega)$ for any $\alpha\in\rr\setminus\{0\}$. }
\end{align*}
Therefore
  $g_{\az}=|Du|^{\az}$  and  then
\begin{align*}
\int_{B(z,r/2)}| D|Du |^{\az}|^2\,dx&\le \liminf_{p\to\fz}
\int_{B(z,r/2)}| D|Du_p|^{\az}|^2\,dx\\
&\le C(1+\az^2)r^{-2}\liminf_{p\to\fz} \int_{B (z,r)}|Du_p|^{2\az}\,dx\\
&=C(1+\az^2)r^{-2}
 \int_{B (z,r)}|Du|^{2\az}\,dx
\end{align*}
as desired.
\end{proof}

As a direct consequence of Lemma \ref{sob-re-x}, we show that
the partial derivative of $|Du|^{\az}$ along $Du$ is zero.
\begin{lem}\label{sob-point}
For any $\alpha\in\rr$, we have
 $D|Du|^\alpha\cdot Du=0$ almost everywhere in $\Omega$.
\end{lem}
\begin{proof}
By $Du_p\neq 0$ in $\Omega$, we
note that
$$D\ln |Du_p|\cdot  Du_p =D|Du_p|\cdot \frac{Du_p}{|Du_p|}=\Delta_\fz ^N u_p=-\frac1{p-2}\Delta u_p\quad{\rm in}\quad \Omega.$$
For any $\phi\in C^\fz_c(\Omega)$, it then follows that
$$\int_\Omega D\ln |Du_p|\cdot  Du_p \phi\,dx
 =-\frac1{p-2}\int_\Omega \Delta u_p\phi\,dx= \frac1{p-2}\int_\Omega Du_p\cdot D\phi\,dx.$$
 Since $Du_p\to Du$ in $C^0(\Omega)$ by Theorem \ref{c1-re} and
 $ \ln |Du_p|\to \ln|Du|$ weakly in $W^{1,2}_\loc(\Omega)$ by Lemma
 \ref{sob-re-x},
 we have
 $$\int_\Omega D\ln |Du|\cdot  Du \phi\,dx=0.$$
 Thus $D\ln |Du|\cdot  Du$=0 almost everywhere in $\Omega$.

 Recall that
 $u\in C^1(\Omega)$ and $Du\neq 0$ in $\Omega$ according to Theorem \ref{c1-re}.
For any $\alpha\in\rr\setminus\{0\}$, it follows from $|Du|^{\az}\in W^{1,2}_{\loc}(\Omega)$ that
 $$\mbox{$D|Du|^\alpha = |Du|^{\alpha} D\ln |Du|$ almost everyhwere in $\Omega$.} $$
Therefore
$D  |Du|^\alpha \cdot  Du$=0 almost everywhere in $\Omega$, which is as desired.
\end{proof}
Now we can finish the proof of Theorem \ref{sob-re}.
\begin{proof}[Proof of Theorem \ref{sob-re}]
It  follows from Lemma \ref{sob-re-x} and Lemma \ref{sob-point}.
\end{proof}

\section{Proof of Theorem \ref{th-geo}}\label{Section8}

 In this section, we   assume that  $u$ is the $\fz$-harmonic potential in $\Omega=\Omega_0\setminus \{x_0\}$, where $\Omega_0$ is a convex domain and
$\overline \Omega_1=\{x_0\}$ with  $x_0\in\Omega_0$.

 To prove Theorem \ref{th-geo}, we need the following key lemma, which when  $n=2$  was already proved by Lindgren-Lindqvist \cite[Corollary 10]{ll19}.
\begin{lem}\label{con-ze}
For any $x\in\Omega$,

$$|Du(x)|\le  \frac{1}{\dist(x ,\partial\Omega_0)}.$$
Moreover, we have
$$\lim_{x\to  x_0}|Du(x)|=\|Du\|_{L^\fz(\Omega_0)}=\frac{1}{
{\rm \dist}(x_ 0,\partial\Omega_0)}.$$
\end{lem}

\begin{proof}
Thanks to the convexity of $\Omega_0$,  we deduce
 $$ \|Du\|_{L^\fz(\Omega_0)}=
 \sup_{x\in\Omega_0}\lip\, u(x)=\lip(u,\Omega_0):=\sup_{x\neq y,x,y\in\Omega_0}\frac{u(x)-u(y)}{|x-y|}.$$
From $\Omega=\Omega_0\setminus\{x_0\}$ and $u\in C^0(\Omega_0)$, it follows that
$\lip(u,\Omega_0) =\lip(u,\Omega) $.
Since $u$ is an absolutely minimizing Lipschitz extension,  we know
$$\lip(u,\Omega) =\lip(u,\partial \Omega)=\sup_{y\in\partial \Omega_0}\frac1{|x_0-y|}=\frac1{\dist(x_0,\partial\Omega_0)}.$$
Therefore  by Theorem \ref{everydiff},
$$|Du(x)|=\lip\, u(x)\le   \|Du\|_{L^\fz(\Omega_0)}=\frac1{\dist(x_0,\partial\Omega_0)}\quad\forall x\in\Omega.$$

Below we show that
$\liminf_{x\to  x_0}|Du(x)|\ge \|Du\|_{L^\fz(\Omega_0)} $.
 To this end,
 using Theorem 1.1 in \cite{swy08} we have
 $$\lim_{x\to x_ 0}\frac{|u(x)-u(x_0)+\|Du\|_{L^\fz(\Omega_0)}|x-x_0||}{|x-x_0|}=0.$$
 Given any $\ez>0$, there is $r_{\ez}<1$ such that
 \begin{align}\label{sin-point}
  \frac{u(x)-u(x_0)}{|x-x_0|}\ge \|Du\|_{L^\fz(\Omega_0)}-\ez\quad
 \forall x\in B(x_0,r_\ez)\backslash\{x_0\}\Subset \Omega_0.
 \end{align}
For each $x\in B(x_0,r_\ez)\backslash\{x_0\}$, by Lemma \ref{low-u} we get
$$|Du(x)|\ge \frac{u(x)-u(x+t\nu_{x})}{t}\quad \forall 0<t<r_{x,\nu_x}
,$$
where $\nu_{x}=-Du(x)/|Du(x)|$ and $u(x+r_{x,\nu_x}\nu_{x})=0$.
 Observe that, given any $t\in(0,\min\{r_\ez,r_{x,\nu_x}\})$,
 via \eqref{sin-point} we obtain
\begin{align*}\liminf_{x\to x_0}|Du(x)|&\ge\liminf_{x\to x_0}\frac{u(x)-u(x+t\nu_{x})}{t}\\
&\ge \frac{u(x_0)- \sup_{|z|=1}u(x_0+tz)}{t}\\
&\ge \|Du\|_{L^\fz(\Omega_0)}-\ez.
\end{align*}
 By the arbitrariness of $\ez>0$, we have  $\liminf_{x\to x_0}|Du(x)|\ge \|Du\|_{L^\fz(\Omega_0)} $ as desired.
\end{proof}

We are ready to prove Theorem \ref{th-geo}.

\begin{proof}[Proof of Theorem \ref{th-geo}]
If $\Omega=B(x_0,r)\setminus \{x_0\}$ for some $r>0$,
then the $\fz$-harmonic potential  $u(x)=1-|x-x_0|/r$. Obviously,   $u\in C^2(\Omega)$ and also $u$ is concave.
Thus (ii) and (iii) follow from (i).

Now we prove (ii) $\Rightarrow$ (i).
Assume   $u\in C^2(\Omega)$.
Obviously, $B(x_0, {\rm dist}(x_0,\partial\Omega_0))\subset\Omega_0$. The proof of (i) is reduced to  proving $
\Omega_0\subset B(x_0,{\rm dist}(x_0,\partial\Omega_0)).$
Given any $x\in \Omega$,  we show the desired
$|x-x_0|< {\rm dist}(x_0,\partial\Omega_0)$ as below.

Using Corollary \ref{streamline}, there exists a curve
$\gz_x\in C^0([0,T_x];\Omega_0)\cap C^1([0,T_x);\Omega) $ for some $0<T_{x}<\fz$ such that
\begin{align}\label{gzx}\frac{d\gz_x(t)}{d t}=Du(\gz_x (t))\quad \forall t\in [0, T_x); \quad \gz_x(0)=x,\quad \gz_x(T_x)=x_0 .
\end{align}
Note that $u\in C^2(\Omega)$,  by $D^2uDu\cdot Du=0$ in $\Omega$ one has
$$\frac{d }{dt}|Du|^2(\gz_x(t))
=D|Du|^2(\gz_x(t))\cdot Du(\gz_x(t))=0$$
and hence
$$|Du(\gz_x(0)|=|Du(\gz_x(t))| \quad \forall t\in [0, T_x).$$
Applying Lemma \ref{con-ze} one gets
$$|Du(x)|=\left|\frac{d\gz_x(t)}{d t}\right|=|Du(\gz_x(t))|=\lim_{s\to T_{x}^-}|Du(\gz_x(s))|= \frac{1}{{\rm dist}(x_0,\partial\Omega_0)} \quad \forall t\in [0, T_x).$$
Thanks to this, for any $t\in (0, T_x)$ we obtain
\begin{align*}
u(\gz_s(t))-u(\gz_x(0))
=\int^{t}_0Du(\gz_x(t)) \cdot \frac{d\gz_x(t)}{d t}\,dt
=\int^{t}_0|Du(\gz_x(t))| \left|\frac{d\gz_x(t)}{d t}\right|\,dt,
 \end{align*}
by $1>u(x_0)-u(x)\ge  u(\gz_s(t))-u(x)>0$ we conclude that
 \begin{align*}1>u(x_0)-u(x)=   \frac{1}{{\rm dist}(x_0,\partial\Omega_0)}
 \int^{t}_0\left|\frac{d\gz_x(t)}{d t}\right|\,dt \ge\frac{ |\gz_x(t)-\gz_x(0)|}{{\rm dist}(x_0,\partial\Omega_0)}\quad \forall t\in (0, T_x).
 \end{align*}
By $\gz_{x}\in C^0([0, T_x])$ and $\gz(T_x)=x_0$, letting $t\to T_{x}^{-}$ we have
$$1>\frac{ |x-x_0|}{{\rm dist}(x_0,\partial\Omega_0)}$$
as desired.

Finally  we prove (iii) $\Rightarrow$ (i).
Assume that $u$ is concave on $\Omega$.  Similarly,  we only need to show $
\Omega_0\subset B(x_0,{\rm dist}(x_0,\partial\Omega_0)).$
We claim that $|Du|$ is a constant.
If so, then by Lemma \ref{con-ze} $|Du|=\frac1{\dist(x_0,\partial\Omega_0)}$.
By Corollary \ref{streamline}, for any $x\in \Omega$ there is a curve
$\gz_x\in C^0([0,T_x];\Omega_0)\cap C^1([0,T_x);\Omega) $ for some $0<T_x<\fz$
so that \eqref{gzx} holds.  Via an argument similarly to above, for all $x\in \Omega_0$ we have $|x-x_0|< \dist(x_0,\partial\Omega)$ as desired.

To prove the claim that $|Du|$ is a constant,  since $|Du|^2\in W^{1,2}_\loc(\Omega)$
we only need to  show that
$D|Du|^2=0$ almost everywhere.
 The concavity of $u$   implies that $u$ is twice differentiable
almost everywhere  on $\Omega$; see  \cite[Theorem 2.3.1]{cs04}.
Denote by $D^2u$ as its Hessian matrix.
Note that  $D^2u$ is symmetric almost everywhere.    Thanks to  Theorem 1.2, $ D|Du|^2\cdot Du=0$ almost everywhere,  and therefore
 $D^2uDu\cdot Du =0$ almost everywhere.
 At each such point  one has $D|Du|^2=\frac12D^2uDu$.
  So it further suffices to show that $D^2uDu=0$ almost everywhere.

Given any point $x$, where  $D^2u$ is symmetric and   $ D^2uDu\cdot Du=0$,
for any $\xi\in\rn\setminus\{0\}$
one has
 $$u(x+t\xi)= u(x)+t Du(x)\cdot \xi+\frac{t^2}2\xi^TD^2u(x)\xi+o(t^2)\quad\mbox{as $t\to0$}.$$
Since
the concavity of $u$ implies that
 $$u(y)\le u(x)+Du(x)\cdot(y-x)\quad\forall y,x\in\Omega.$$
By setting $y=x+t\xi$ for sufficiently small $t$, and letting $t\to 0$ we conclude that
 \begin{align*}
D^2u(x) \xi\cdot \xi\le \lim_{t\to 0}\frac{o(t^2)}{t^2}=0
\quad \forall \xi\in \rr^n.
\end{align*}
Taking $\xi=  \frac1\ez  Du(x)+ \ez D^2u(x)Du(x)$ with $\ez>0$, at $x$ we get
 \begin{align*}
 \frac1{\ez^2}  D^2uDu\cdot Du+
2|D^2uDu|^2+ \ez^2  D^2uD^2uDu \cdot D^2uDu\le
0.
\end{align*}
Since $D^2uDu\cdot Du=0$ at point $x$ we get
 \begin{align*}
2|D^2uDu|^2+\ez^2D^2uD^2uDu \cdot D^2uDu\le 0.
\end{align*}
Letting $\ez\to0$, at $x$ we have $|D^2uDu|^2=0$  as desired.

\end{proof}

\section{Proof of Theorem \ref{secsob} }\label{Section9}
In this section we prove Theorem \ref{secsob} by using
Theorem \ref{sob-re} and  a fundamental  inequality
in Lemma \ref{id1}.

Let us start by proving the following upper bound of $D^2u_p$  for
the $p$-harmonic potential $u_p$ in $\Omega$.

\begin{lem}\label{sec-sobup}
Let $p\in(4,\fz)$. For all $B(z,r)\Subset \Omega$ we have
\begin{align*}
\int_{B(z,r)}| D^2u_p|\xi^2\,dx\le 2\int_{B(z,r)}-\bdz u_p\xi^2\,dx
+2\int_{B(z,r)}| D|Du_p||\xi^2\,dx\quad \forall \xi\in C^\fz_c(B(z,r)).
\end{align*}
\end{lem}
\begin{proof}
 Since $Du_p\neq 0$ on $\Omega $, recalling that Lemma \ref{id1} give us
$$| D^2u_p|^2-(\bdz u_p)^2\le 2[| D|Du_p||^2-\bdz_\fz^N u_p \bdz u_p]\quad {\rm in}\quad
\Omega.$$
Using equation $-\bdz_\fz^N u_p =\frac1{p-2}\bdz u_p$  in $\Omega$, we have
$$| D^2u_p|^2\le [1+\frac{2}{p-2}](\bdz u_p)^2+2| D|Du_p||^2
\le 2[(\bdz u_p)^2+| D|Du_p||^2]\quad {\rm in}\quad
B(z,r)\Subset \Omega,$$
which immediately implies
$$| D^2u_p|\le 2|\bdz u_p|+2| D|Du_p||\quad {\rm in}\quad
B(z,r),\quad \forall p\in (4,\fz).$$
Multiplying both sides by a test function $\xi^2$ with $\xi\in C^\fz_c(B(z,r))$,
and integrating over on $B(z,r)$,  we get
\begin{align*}
\int_{B(z,r)}| D^2u_p|\xi^2\,dx\le 2\int_{B(z,r)}|\bdz u_p|\xi^2\,dx
+2\int_{B(z,r)}| D|Du_p||\xi^2\,dx.
\end{align*}
Since
$-\bdz u_p\ge 0$ on $\Omega $ (see \cite{l77}), we obtain the desired inequality.
\end{proof}

  By Lemma \ref{sob-es}, we  have  the following upper bound uniformly in all $p>4$.
\begin{lem}\label{sec-sobup-2}
For $p\in(4,\fz)$, there exists a universal constant $C$ such that
\begin{align*}
 \int_{B(z,r/2)}| D^2u_p|\,dx
&\le C\frac1r\int_{B(z,r)}|Du_p|\,dx \\
&\quad\quad
+C\frac1r \max_{B(z,r)}|Du_p|^2\left(\int_{B(z,r)}|Du_p|^{-2}\,dx\right)^{\frac 12}  \quad
\mbox{ whenever $B(z,r)\Subset \Omega $.}
\end{align*}

\end{lem}

\begin{proof} Thanks to   Lemma \ref{sec-sobup} and Lemma \ref{sob-es},
we only need to bound $\int_{B(z,r)}-\bdz u_p \xi^2\,dx$.
 Via integration by parts we obtain
\begin{align*}
2\int_{B(z,r)}-\bdz u_p \xi^2\,dx&=-2\int_{B(z,r)}\bdz u_p\xi^2\,dx=4\int_{B(z,r)}Du_p\cdot D \xi \xi\,dx\quad \forall \xi\in C^\fz_c(B(z,r)).
\end{align*}
By choosing a suitable cut-off function $\xi$,
one has
$$2\int_{B(z,r)}-\bdz u_p\xi^2\,dx\le \frac1r\int_{B(z,r)}|Du_p|\,dx.$$
\end{proof}

Letting  $p\to\fz$ we conclude the following from Lemma \ref{sec-sobup-2}.
We write $\|\mu\|$ as the variation measure of a signed Radon measure $\mu$. If $\mu$ is nonnegative, then $\|\mu\|=\mu$.
\begin{lem}\label{w21}
The distributional second order derivatives $\mathcal D^2u$ are Radon measures   satisfying
\begin{align}\label{du2-eq}
\int_\Omega \langle \mathcal D^2u Du, Du\rangle\xi \,dx=0
 \quad \forall  \xi\in C^\fz_c(\Omega) ,
\end{align}
and also $D^2u_p\to \mathcal D^2u$ weakly in the sense of measure.

It holds that $Du\in BV_\loc(\Omega)$, that is,
 the distributional derivatives   $\mathcal D [Du]$ of $Du$ are Radon measures, and that
   $\mathcal D [Du]=\mathcal D^2u$.

 Moreover, there exists a universal constant $C$ such that
\begin{align}\label{w21-eq1}
 \| \mathcal{D}^2u\|(B(z,r/2))
&\le C\frac1r\int_{B(z,r)}|Du|\,dx \nonumber\\
&\quad\quad
+C\frac1r \max_{B(z,r)}|Du|^2\left(\int_{B(z,r)}|Du|^{-2}\,dx\right)^{\frac 12}  \quad
\mbox{ whenever $B(z,r)\Subset \Omega $,}
\end{align}
and also
\begin{align}\label{w21-eq2}
 \| \mathcal{D}^2u\|(B(z,r/2)) \le 2[-\bdz u](B(z,r))
+2\int_{B(z,r)}| D|Du || \,dx \quad
\mbox{ whenever $B(z,r)\Subset \Omega $,}
\end{align}
where $-\bdz u$ is a nonnegative  Radon measure.
\end{lem}

\begin{proof}
Applying Lemma \ref{sec-sobup-2} and Lemma \ref{un-bdup}, we know that
$D^2u_p\in L^1_\loc(\Omega)$ uniformly in all $p>4$.
Via the compactness of the space $\mathcal M(K)$ of Radon measures in any compact subset $K \Subset \Omega$,
for all $1\le i,j\le n$
we can find a locally finite  Radon measure $\mu_{ij}$ so that
 $ \frac{\partial^2 u_p}{\partial x_i\partial x_j}$ weakly converges to  $\mu_{ij}$ in each $K \Subset \Omega$ in the sense of measure,
that is,
$$\int_\Omega \frac{\partial^2 u_p}{\partial x_i\partial x_j}\phi\,dx\to \mu_{ij}(\phi)\quad\forall \phi\in C^{\fz}_c(\Omega).$$
On the other hand, since $ u_p\to  u$ in $C^0(\Omega)$, we know that
$$\int_\Omega \frac{\partial^2 u_p}{\partial x_i\partial x_j}\phi\,dx=\int_\Omega u_p\frac{\partial^2 \phi}{\partial x_i\partial x_j} ,dx\to
\int_\Omega u \frac{\partial^2 \phi}{\partial x_i\partial x_j}  dx  \quad\forall \phi\in C^{\fz}_c(\Omega).$$
That is, $\mu_{ij}$ coincides with the distributional second order partial derivatives  $\mathcal D_{ij}u$
of $u$.
Note  that $\mathcal D_{ji}u=  \mathcal D_{ij}u$, we have $\mu_{ij}=\mu_{ji}$.

Since
$$\int_\Omega
\mathcal D_j u_{x_i} \phi\,dx=-\int_\Omega  u_{x_i} \phi_{x_j}\,dx=
\int_\Omega  u \phi_{x_jx_i}\,dx=\int_{\Omega}\mathcal D_{x_ix_j}u\phi\,dx,$$
we know that  the distributional  derivative $\mathcal D_j u_{x_i}=\mu_{ij}$ is also
a Radon measure, that is, $u_{x_i}\in BV_\loc(\Omega)$.

Recall that $Du_p\to Du$ in $C^0(\Omega)$ in Theorem
\ref{c1-re} and $-\bdz u_p\ge 0$ in $\Omega$, then passing to the
limit $p\to \fz$ in Lemma \ref{sec-sobup} and Lemma \ref{sec-sobup-2} we get
\eqref{w21-eq1} and \eqref{w21-eq2} as desired. In particular, we have that
$-\bdz u$ is a nonnegative Radon measure. Similarly,  using equation
 $D^2u_p Du_p\cdot Du_p=-\frac {1}{p-2}\bdz u_p$ we conclude \eqref{du2-eq} follows by
 letting $p\to \fz$.

\end{proof}
We now come to prove Theorem \ref{secsob}.

\begin{proof}[Proof of Theorem \ref{secsob}]
This follows from Lemma \ref{w21}.
\end{proof}

\section{Proof of Theorem \ref{twice-re}}\label{Section10}

In this section we prove Theorem \ref{twice-re}.
We assume that dimension $n=2$.

We start by showing the following.
\begin{lem}\label{twice-le}
For all $\alpha\in\rr$, $|Du|^\alpha$ is differentiable almost everywhere in $\Omega$.
For $i=1,2$, $u_{x_i}$ and $u^2_{x_i}$ are differentiable almost everywhere in $\Omega$.
\end{lem}

The  following  can be concluded from Rademacher's theorem \cite{r19}; see also
Stepanov in \cite{s25}.
\begin{lem}\label{Stepanoff}
Let $U\subset \rr^n$ be any domain.
A function $g: U\to \rr$ is differentiable almost everywhere in $U$ if and only if
$\lip \, g(x)
<\fz$   for  almost   all $ x\in U.$

\end{lem}

\begin{proof}[Proof of Lemma \ref{twice-le}]
We first show that $|Du|^{\az}$ is differentiable almost everywhere in $\Omega$ for all
$\az\in \rr$.
Recall that Theorem \ref{c1-re} gives us
$u\in C^1(\Omega)$ and $|Du|\neq 0$ in $\Omega$. It suffices to prove
the almost everywhere differentiability of $|Du|^2$. 
Thanks to  Lemma \ref{Stepanoff}, we only need  to prove that
$\lip( |Du|^2)(z)<\fz$ for almost all $z\in \Omega$. 
To this end, using \eqref{m-e1} in
Lemma \ref{mn-u2}, we have
\begin{align*}\lip( |Du|^2)(z)&=\lim_{r\to0}\sup_{y\in B(z,r)}\frac{ |Du(y)|^2-|Du(z)|^2}{r}\\
&\le
\limsup_{r\to0}\osc_{B(z,r)} \frac{|Du|^2}r\\
&\le \limsup_{r\to0}\left(\mint-_{B(z,r)}|D|Du|^2|^2\,dx\right)^{1/2},
\end{align*}
which is finite whenever $z$ is  a Lebesgue point of  $|D|Du|^2|^2$.
Since  $|D|Du|^2|^2\in L^1_\loc(\Omega)$,  its Lebesgue points is dense in $\Omega$. Thus
$\lip( |Du|^2)(z)<\fz$ for almost all $z\in \Omega$ as desired.

Next, we show the almost everywhere differentiability of  $u_{x_i}$,
which gives the almost everywhere differentiability  of $u^2_{x_i}$  in an obvious way.
 Similarly to above,    it suffices to prove that
$\lip( u_{x_i})(z)<\fz$ for almost all $z\in \Omega$. By  \eqref{m-e2} in Lemma
 \ref{mn-u2} and Lemma \ref{w21}, we have
\begin{align}\label{x10.x1}
\lip( u_{x_i})(z)&=\limsup_{r\to0}\sup_{y\in B(z,r)}\frac{ u_{x_i}(y)-u_{x_i}(z)}{r}\nonumber\\
&\le \limsup_{r\to0}\frac{\osc_{B(z,r)}  u_{x_i}}r\nonumber\\
&\le \limsup_{r\to0} \frac{ \| D^2u\|(B(z,r/2)) }{r^2}\nonumber\\
&\le  C \limsup_{r\to0}  \frac{\|-\bdz u\|(B(z,r)) }{r^2}
+C\limsup_{r\to0}\mint{-}_{B(z,r)}| D|Du ||\,dx.
\end{align}
Thanks to
$| D|Du ||\in L^1_\loc(\Omega)$,  at its Lebesgue points and hence almost all points in $\Omega$, one has
\begin{align}\label{x10.x2}
 \limsup_{r\to0}\mint-_{B(z,r)}| D|Du ||\,dx<\fz \quad \mbox{for almost all $z\in \Omega$.}
 \end{align}
Thanks to \eqref{x10.x1} and \eqref{x10.x2},  in order to show
 $\lip( u_{x_i})(z)<\fz$ for almost all $z\in \Omega$,
it suffices to prove that \begin{align}\label{x10.x3} \limsup_{r\to0}  \frac{\|-\bdz u\|(B(z,r)) }{r^2} <\fz\ \mbox{ for almost all $z\in \Omega$.  } \end{align}

  Below we prove \eqref{x10.x3}. Recall that $Du\in BV_{\loc}(\Omega)$ and the distributional $[-\bdz u]$ is a nonnegative Radon measure.
  We therefore write $[-\bdz u]$ as the absolutely continuous part $[-\bdz u]_{ac}$ and the singular part $[-\bdz u]_{s}$ with respect to the Lebesgue measure, that is,
$$[-\bdz u]=[-\bdz u]=[-\bdz u]_{ac}+[-\bdz u]_{s};$$
see \cite[Chapter 6]{eg15}.
 By \cite[Theorem 6.1]{eg15} we have
\begin{align}\label{sin-pa}
\lim_{r\to 0}\frac{[-\bdz u]_s(B(z,r))}{r^2}=0\quad \mbox{for almost all $z\in\Omega$}.
\end{align}
Therefore
\begin{align*}
\lim_{r\to0}\frac{[-\bdz u](B(z,r))}{r^2}=\lim_{r\to0}\frac{[-\Delta u]_{ac}(B(z,r))}{r^2}\quad \mbox{for almost all $z\in\Omega$}.
\end{align*}
Denote by  $g\in L^1_\loc(\Omega)$  the Radon-Nikodym derivative of
$[-\bdz u]_{ac}$, that $[-\bdz u]_{ac}=g\,dx$.
At  any Lebesgue points $z$,
we have
$$\frac{[-\Delta u]_{ac}(B(z,r))}{r^2}=\lim_{r\to 0}\mint{-}_{B(z,r)}g(y)\,dy=g(z).$$
Thanks to   $g\in L^1_\loc(\Omega)$  and
the density of Lebesgue points of $g$ in $\Omega$  we attain
 \begin{align*}
\lim_{r\to0}\frac{[-\Delta u]_{ac}(B(z,r))}{r^2}
=g(z)<\fz\quad \mbox{for almost all $z\in\Omega$}.
\end{align*}
From this and \eqref{sin-pa} we conclude  \eqref{x10.x3} as desired.
\end{proof}

\begin{rem}\rm
Since $|Du|^2$ is monotone and
$|Du|^2\in W^{1,2}_{\loc}(\Omega)$,  by Onninen \cite[Theorem 1.2]{o20},
 $|Du|^2$ is differentiable almost everywhere.
 Here we give a direct proof  via Lemma \ref{mn-u2}.
 \end{rem}

Next, we show the following  property for the absolutely continuous part of Radon measure
$\mathcal D^2u$.
\begin{lem}\label{ty-f1}
 For $1\le i,j\le 2$, the absolutely continuous part of the measure
 $\mathcal D^2_{ji}u$  is given by $  u_{x_ix_{j}} \,dx$, that is, $$[\mathcal D^2u]_{ac}=[\mathcal D(Du)]_{ac}=D^2u\,dx.$$
In particular, $D^2u\in L^1_\loc(\Omega)$ is symmetric almost everywhere, that is,  $u_{x_1x_2}= u_{x_2x_1 }$  almost everywhere in $\Omega$.
\end{lem}

\begin{proof}
Since $u_{x_i}\in BV_{\loc}(\Omega)$, the measure
$\mathcal D_ju_{x_i}$ is decomposed as the absolutely  continuous part   $[\mathcal D_ju_{x_i}]_{ac} $ and
the singular part $[\mathcal D_ju_{x_i}]_{s} $ with respect to the Lebesgue measure.
Denote by $ g_{ij}\in L^1_\loc(\Omega)$ the Radon-Nikodym derivative of $[\mathcal D_ju_{x_i}]_{ac} $ with respect to the Lesbesgue measure,
that  is,
 $[\mathcal D_ju_{x_i}]_{ac} =g_{ij}\,dx$.
 Since  $\mathcal D_1u_{x_2}=\mathcal D_{12}u= \mathcal D_{21}u=\mathcal D_2u_{x_1}$, we know that
$ g_{12}=g_{21}$ almost everywhere in $\Omega$.
 Let $z$ be any Lebesgue point of $\{g_{ij}\}_{j=1,2}$ so that  $ g_{12}(z)=g_{21}(z)$ and
  $Du $ is differentiable at $z$. Note that the set of all such $z$ is dense in $\Omega$.
Applying \cite[Theorem 6.1]{eg15} we have
\begin{align}\label{ex1}
\left(\mint{-}_{B(z,r)}|u_{x_i}(x)- u_{x_i}(z)-\sum_{j=1}^2g_{ij}(z)\cdot (x_j-z_j)|^2\,dx\right)^{\frac 12}=o(r).
\end{align}
Since $u_{x_i}$ is differentiable  at $z$,  we have
\begin{align}\label{ex3}
\left|u_{x_i}(x)- u_{x_i}(z)-\sum_{j=1}^2u_{x_ix_j} (z) (x_j-z_j)\right|=o(r).
\end{align}
Thus
\begin{align*}
\left(\mint{-}_{B(z,r)}\sum_{i=1}^2|\sum_{j=1}^2[g_{ij}(z)- u_{x_ix_j}(z)] (x_j-z_j)|^2\,dx\right)^{\frac 12}=o(r).
\end{align*}
This implies that
$g_{ij}(z)=  u_{x_ix_j} (z)$ for all possible $ i,j$.
\end{proof}

It follows from Lemmas \ref{twice-le} and \ref{ty-f1} that
$u$ is twice differentiable almost everywhere in $\Omega$.
\begin{lem}\label{ty-f2}
It holds that
$u$ is twice differentiable almost everywhere in $\Omega$, that is, for almost all $z\in\Omega$,
$$\lim_{r\to0}\sup_{x\in B(z,r)}\frac{\left|u(x)-u(z)-Du(z)\cdot (x-z)-\frac12 (x-z)^T
\cdot D^2u(x)(x-z) \right|}{r^2}=0.$$
\end{lem}

\begin{proof}
Let   $z\in\Omega$  be any point where
 $Du$ is differentiable at $z$. Without loss of generality, we assume that  $z=0$.
   Writing
 $$u(x)-u(0)=\int^1_0x\cdot Du(tx)\,dt,$$
 we have
\begin{align}\label{taxp0}
u (x)-u(0)+Du(0)\cdot x=\int^1_0 x \cdot [ Du (t x)-Du (0)]\,dt.
\end{align}
Since $D u $ is differentiable at $0$,  for each $\ez\in (0,1)$ there is $r_\ez>0$ such that  for $ r<r_\ez$ one has
\begin{align}\label{taxp}
|D  u(y)-D u (0)-  D^2u(0)y|\le \ez r \quad \forall y\in B(0,r).
\end{align}
If $x\in B(0,r)$ and $0<t<1$, applying \eqref{taxp} to $y=tx$ we have
 $$|x\cdot Du(tx)-x\cdot Du (0)- x\cdot D^2u(0)tx]|\le \ez r^2.$$
 Thus
  $$\left|\int_0^1[x\cdot Du(tx)- x\cdot Du (0) - x\cdot D^2u(0)tx]\,dt\right|\le \ez r^2$$
  and hence by \eqref{taxp0} we get
  $$\left|u (x)-u(0)+Du(0)\cdot x-\frac12 x^T\cdot D^2u(0)x\right|\le \ez r^2$$
  as desired.
     \end{proof}

\begin{proof}[Proof of Theorem \ref{twice-re}]
Thanks to Lemma \ref{ty-f1} and Lemma \ref{ty-f2}, we only need to
prove that $-D^2uDu\cdot Du=0$, \eqref{Taylor-exp} and \eqref{Taylor-exp2}  hold almost everywhere in $\Omega$.

Since $|Du|^2, u^2_{x_1},u^2_{x_2}$ are differentiable almost everywhere,
 we know that
\begin{align*}
D|Du|^2=D[u_{x_1}^2+u_{x_2}^2]=2(u_{x_1x_1}u_{x_1}+u_{x_2x_1}u_{x_2},u_{x_1x_2}u_{x_1}
+u_{x_2x_2}u_{x_2})=2D^2uDu
\end{align*}
almost everywhere.
Recall that $|Du|^2\in W^{1,2}_\loc$ and
 $D|Du|^2\cdot Du=0 $ almost everywhere, we have
  $-D^2uDu\cdot Du=0$ almost everywhere.

Next we show
\eqref{Taylor-exp} and \eqref{Taylor-exp2} at any point $x$,  where
$u$ is twice differentiable at $x$ and $-D^2u Du \cdot Du=0$ at $x$.
Without loss of generality we write $x=0$.
By Taylor's expansion, we have
$$u(z)=u(0)+Du(0)\cdot z+\frac 12   D^2u(0)z\cdot z+o(|z|^2)\quad\mbox{as $|z|\to 0$.}$$
Since $Du\neq 0$ on $\Omega$, choosing $z=\pm h Du(0) 
$, we get
$$u(\pm h Du(0))=u(0)\pm h|Du(0)|+\frac 12 h^2 D^2u(0)Du(0)\cdot Du(0)
+o(h^2)\quad\mbox{as $|h|\to 0$ },$$
and hence
$$u(0)=\frac12\left[u\left( h Du(0)
\right)+u\left(- h Du(0)
 \right)\right]+o(h^2)\quad\mbox{as $|h|\to 0$ },$$
which gives \eqref{Taylor-exp} at $x=0$ as desired.

 To get  \eqref{Taylor-exp2} at $x=0$,
choose $|x^\pm_\ez|=\ez$ so that
$$\max_{\overline{B(0,\ez)}}u=u(x^+_\ez)\ \mbox{and}\  \min_{\overline{B(0,\ez)}}u=u(x^-_\ez).$$
Recalling that $u\in C^1(\Omega)$ and  $Du\neq 0$ on $\Omega$,
and noting  that  at $ x^\pm_\ez$ the tangential derivatives of $u$ along $\partial B_\ez$  are zero,
we have
\begin{align}\label{Lagrangian}
\frac{Du(x^\pm_\ez)}{|Du(x^\pm_\ez)|}=\pm \frac{x ^\pm_\ez}{|x^\pm_\ez|}.
\end{align}
Observe that $\frac{Du}{|Du|}$ is differentiable at $0$ due to Lemma \ref{twice-le}, we have
\begin{align}\label{diff-dudux}
 \frac{Du(x^\pm_\ez)}{|Du(x^\pm_\ez)|}
  &= \frac{Du(0)}{|Du(0)|}+O(\ez).
\end{align}
On the other hand, by Taylor's expansion,  one has
\begin{align}\label{mean-x1}
\max_{\overline{B(0,\ez)}}u+\min_{\overline{B(0,\ez)}}u
&\le u(x^+_\ez)+u(-x^+_\ez)=2u(0)+D^2u(0)x^+_\ez\cdot x^+_\ez+o(\ez^2)
\end{align}
and
\begin{align}\label{mean-x2}
\max_{\overline{B(0,\ez)}}u+\min_{\overline{B(0,\ez)}}u
&\ge u(x^-_\ez)+u(-x^-_\ez)=2u(0)+D^2u(0)x^-_\ez\cdot x^-_\ez+o(\ez^2).
\end{align}
Using \eqref{Lagrangian}, \eqref{diff-dudux} and $D^2u Du\cdot Du=0$ at point $0$ we have
\begin{align*}
D^2u(0)x^\pm_\ez\cdot x^\pm_\ez&=
 \ez^2  D^2u(0)\frac{Du(x^\pm_\ez)}{|Du(x^\pm_\ez)|}\cdot \frac{Du(x^\pm_\ez)}{|Du(x^\pm_\ez)|}\\
 &= \ez^2  D^2u(0)[\frac{Du(0)}{|Du(0)|}+O(\ez)]\cdot  [\frac{Du(0)}{|Du(0)|}+O(\ez)]\\
 &= \ez^2  D^2u(0) \frac{Du(0)}{|Du(0)|}\cdot  \frac{Du(0)}{|Du(0)|}+O(\ez^3) \\
 &=o(\ez^2).
\end{align*}  From \eqref{mean-x1} and \eqref{mean-x2}, it follows that
$$
2u(0)+o(\ez^2)\le \max_{\overline{B(0,\ez)}}u+\min_{\overline{B(0,\ez)}}u
 \le  2u(0)+ o(\ez^2),
$$
which gives \eqref{Taylor-exp} at $x=0$ as desired.
\end{proof}

\renewcommand{\thesection}{Appendix A}
 \renewcommand{\thesubsection}{ A }
\newtheorem{lemapp}{Lemma \hspace{-0.15cm}}
\newtheorem{corapp}[lemapp] {Corollary \hspace{-0.15cm}}
\newtheorem{remapp}[lemapp]  {Remark  \hspace{-0.15cm}}
\newtheorem{defnapp}[lemapp]  {Definition  \hspace{-0.15cm}}
\renewcommand{\theequation}{A.\arabic{equation}}

\renewcommand{\thelemapp}{A.\arabic{lemapp}}

\section{Proof of Corollary \ref{streamline}}

Theorem \ref{c1-re}  allows us to borrow some idea from  Lindgren-Lindqvist \cite{ll21} to prove
 Corollary \ref{streamline}. We give the details for reader's convenience.

\begin{proof}[Proof of Corollary \ref{streamline}]
Fix any $x \in \Omega=\Omega_0\backslash \overline \Omega_1$.  We split the proof into 2 steps.

{\it Step 1.}
Given any $p>2$, since $u_p\in C^\fz(\Omega)$, by
 \cite[Corollary 2.3]{ghkl18}
there exists a unique solution $\gz^p_x \in C^1([0,T_x^p),\Omega)$ for some $T^p_x\in(0,\fz]$ to the problem
\begin{align*}
\frac{d \gz^p_x(t)}{d t}&=Du_p(\gz^p_x(t)) \quad \forall t\in [0,T_x^p);\ \gz^p_x(0)=x,
\end{align*}
  where     $[0,T_x^p)$ is the   maximal time interval.
 Set
$$\Phi_x^p(t):=u_p(\gz^p_x(t))\quad \forall t\in [0, T_x^p).$$
Observe that
$$
\frac{d \Phi_x^p(t)}{dt}=|Du_p(\gz_x^p(t))|^2\ge0\quad \forall t\in [0, T_x^p).$$
Thus, $\Phi_x^p$ is nondecreasing  in $[0,T_x^p)$.
Moreover, we compute that
\begin{align*}  \frac{d^2 \Phi_x^p(t)}{dt^2}
&=\frac{d}{dt}|Du_p(\gz_x^p(t))|^2
 =2D^2 u_p(\gz_x^p(t))Du_p(\gz_x^p(t))\cdot \frac{d \gz^p_x(t)}{d t}
 =2 \Delta_\fz u_p(\gz_x^p(t)).\end{align*}
Since
$$\mbox{$\Delta_pu_p=|Du_p|^{p-2}(\Delta u_p+(p-2)\Delta^N_\fz u_p)=0$, $Du_p\ne0$
 and $-\Delta u_p\ge0$ in $\Omega$},$$
 it follows that
 \begin{align*}  \frac{d^2 \Phi_x^p(t)}{dt^2}
&=-\frac{2}{p-2}|Du_p(\gz^p_x(t))|^2
\bdz u_p(\gz^p_x(t))\ge 0 ,
\end{align*}
which means that $\Phi_x^p$ is convex on $[0,T_x^p)$ and   $|Du_p\circ \gz^p_x|$ is nondecreasing in $[0, T_x^p)$.

Now using the fact that $|Du_p\circ \gz^p_x|$ is nondecreasing in $[0, T_x^p)$,
via \eqref{low-grup} in Lemma \ref{low-up} we conclude that
$$
|Du_p(\gz_x^p(t))|\ge |Du_p(\gz_x^p(0))|=|Du_p(x)|\ge \frac{u_p(x)}{{\rm diam}(\Omega_0)} \quad \forall t\in [0, T_x^p).$$
 Since $u_p\to u$ in $C^0(\overline \Omega)$ and $u(x)>0$,  there exists a fixed constant $p_{x}>2$ such that
$$
\frac{d \Phi_x^p(t)}{dt}\ge\left [\frac 12\frac{u(x)}{{\rm diam}(\Omega_0)}  \right]^2>0
 \quad \forall t\in [0, T_x^p)\quad \forall p>p_{x}.$$
This leads to
\begin{align}\label{max-time0}
1\ge  \Phi_x^p(t)-\Phi_x^p(0)
=\int^t_0 \frac{d \Phi_x^p(s)}{ds}  \,ds
\ge t\left[\frac12\frac{u(x)}{{\rm diam}(\Omega_0)} \right]^2 \quad \forall 0<t<T_x^p,\quad \forall
p>p_{x},
\end{align}
and hence
\begin{align}\label{max-time}
T_x^p\le \left[\frac12\frac{{\rm diam}(\Omega_0)}{u(x)}\right]^2<\fz \quad \forall p>p_{x}.
\end{align}

 Next we  extend $\gz_x^p$ to $[0,T_x^p]$ by setting
 $\gz_x^p(T_x^p)=\lim_{t\to {T_x^p}^{-}} \gz_x^p(t).$ This comes from the fact that
 \begin{align}\label{holder}| \gz_x^p(t)-\gz_x^p(s)|& =\int_s^t|\frac{d \gz^p_x(\dz)}{d t}|\,d\dz\nonumber\\
 &=
\int_s^t|Du_p(\gz^p_x(t)) |\,d\dz\nonumber\\
&\le (t-s)^{1/2} \left(\int_s^t|Du_p(\gz^p_x(\dz)) |^2\,d\dz\right)^{1/2}\nonumber\\
&\le (t-s)^{1/2}\left(\int_s^t \frac{d \Phi_x^p(\dz)}{dt}\,d\dz\right)^{1/2}\nonumber\\
&\le (t-s)^{1/2}(\Phi_x^p(t)-\Phi_x^p(s))^{1/2}\nonumber\\
&\le (t-s)^{1/2} \quad \forall 0\le s<t<T_x^p.
 \end{align}
Note that $\gz_x^p(T_x^p)\in\partial\overline \Omega_1$ and hence $u_p(\gz_x^p(T_x^p))=1$.
 Indeed, thanks to $u_p(\gz_x^p(T_x^p))\ge u_p(x)$ we know that $\gz_x^p(T_x^p)\notin \partial \Omega_0$; since  $[0,T_x^p)$ is the maximal interval,
$ \gz_x^p(T_x^p)\notin\Omega$.

{\it Step 2.}
By \eqref{max-time},  $ T_x^p $ is bounded uniformly in $p>p_{x}$,  and hence,
 up to some subsequence we may assume that $T_x^p\to T_x$ as $p\to\fz$.
By \eqref{holder}, $ \gz_x^p\in C^{1/2}([0,T_x^p])$   uniformly in
$p>p_x$.   Thus $\{\gz_x^p\}_{p>p_x}$
is  uniformly bounded and equal continuous in $[0,T_x]$,
  where if $T_x^p<T_x$ we let $\gz_x^p(t)=\gz_x^p(T_x^p)$
for $t\in[T_x^p,T_x]$.
 Therefore, we can find  a curve $\gz_x\in C^{0,1/2}([0,T_x])$  with
 $\gz_x(0)=x$ and $\gz_x(T_x)\in\partial\Omega_1$, so that
 $\gz_x^p\to \gz_x$  in $C^{1/2}([0,T_x^p])$   as $p\to\fz$ (up to some subsequence).

Write $\Phi_x(t)=u(\gz_x(t))$ for all $ t\in [0,T_x]$.
  Then $\Phi_x^p\to \Phi_x$ in $C^0([0,T_x])$.
Given any $s,t\in [0,T_x)$,   due to $T_x^p\to T_x$,  we know that
$s,t\in[0,T_x^p]$ for all sufficiently large $p$ (up to some subsequence).
Thus
 \begin{align*}
 \Phi_x(t)-\Phi_x(s)=\lim_{p\to\fz}[\Phi_x^p(t)-\Phi_x^p(s)]=\lim_{p\to\fz}\int_s^t\frac{d\Phi_x^p(\dz)}{d\dz}\,d\dz\ge (t-s)\left[\frac{u(x)}{2\diam(\Omega_0)}\right]^2,
\end{align*}
 which implies that
 $\Phi_x$ is strictly increasing.
 Moreover,   $\Phi_x$ is convex on $[0,T_x]$ since
\begin{align*}\Phi_x(\lz s+(1-\lz )t)&= \lim_{p\to\fz}
  \Phi_x^p(\lz s+(1-\lz )t)\\
  &\le  \lim_{p\to\fz}
 [ \lz\Phi_x^p( s)+(1-\lz )\Phi_x^p (t)]\\
 &= [ \lz\Phi_x ( s)+(1-\lz )\Phi_x  (t)]
 \quad\forall 0<s<t<T_x,\lz\in(0,1),\end{align*}

For  any $0 <t <T_x$, thanks to the strictly  increasing property of $\Phi_x$, we know that  $\gz_x([0,t])\subset \Omega$ and hence
 $\gz_x^p([0,t])$ is contained in
 a neighborhood $U\Subset\Omega$ of  for all sufficiently large $p$  (up to some subsequence).
 Thus by Theorem \ref{c1-re} we have $Du_p\circ\gz_x^p \to Du\circ \gz_x $ in $C^0([0,t])$ as $p\to\fz$  (up to some subsequence).
  Since $|Du_p\circ \gz^p_x |$ is nondecreasing in $[0, T_x^p)$, we know that  $|Du\circ \gz_x|$ is also nondecreasing in $[0,T_x)$.  Moreover,
for any $0\le s<t<T_x$, one has
  \begin{align*}
 \gz_x(t)-\gz_x(s)&=\lim_{p\to\fz}[\gz_x^p(t)-\gz_x^p(s)]\\
 &=\lim_{p\to\fz}\int_s^t\frac{d\gz_x^p(\dz)}{d\dz}\,d\dz\\
 &= \lim_{p\to\fz}\int_s^t Du_p(\gz_x^p(\dz))\,d\dz \\
 &=\int_s^t Du (\gz_ x(\dz))\,d\dz,
\end{align*}
which implies that $\gz_{x}\in C^1([0,T_x))$ with  $\frac{d\gz_x (t)}{\,dt}=Du(\gz_x(t))$ for all
$t\in (0,T_x)$.  Since $Du\in L^\fz(\Omega)$, we also know that
$\gz_x\in C^{0,1}([0,T_x])$.
\end{proof}
\renewcommand{\thesection}{Appendix B}
\newtheorem{lembpp}{Lemma \hspace{-0.15cm}}
\newtheorem{thmbpp}[lembpp] {Theorem \hspace{-0.15cm}}
\newtheorem{corbpp}[lembpp] {Corollary \hspace{-0.15cm}}
\newtheorem{rembpp}[lembpp]  {Remark  \hspace{-0.15cm}}
\newtheorem{defnbpp}[lembpp]  {Definition  \hspace{-0.15cm}}
\renewcommand{\theequation}{B.\arabic{equation}}

\renewcommand{\thelembpp}{B.\arabic{lembpp}}

\section{Monotonicity in dimension $n=2$}
In this section we assume $n=2$. For $p\in(2,\fz)$, denote by $u_p$ be the $p$-harmonic potential in any given convex ring $\Omega$.

Note that $|Du_p|\neq 0$ on $\Omega$.
The following  monotonicity  for  $u_p$  is a direct consequence of
 the quasi-regular mapping due to Bojarski-Iwaniec \cite{bi87}.  For the readers of convenience, we provide the proof in this Appendix.
\begin{lembpp}\label{mn-up}
For all $B(z,r)\Subset \Omega$, we have
\begin{align}\label{m1}
\max_{\overline {B(z,r)}}|Du_p|=\max_{\partial B(z,r)}|Du_p|,\quad \min_{\overline {B(z,r)}}|Du_p|=
\min_{\partial B(z,r)}|Du_p|
\end{align}
and
\begin{align}\label{m2}
\max_{\overline {B(z,r)}}\frac{\partial u_p}{\partial x_i}=\max_{\partial B(z,r)}\frac{\partial u_p}{\partial x_i},\quad \min_{\overline {B(z,r)}}\frac{\partial u_p}{\partial x_i}=
\min_{\partial B(z,r)}\frac{\partial u_p}{\partial x_i} \quad\mbox{ for $ i=1,2$}.
\end{align}
\end{lembpp}
\begin{proof}
Recall that $u_p\in C^\fz(\Omega)$ and $Du_p\neq 0$ in $\Omega$. Using equation
$-(p-2)\bdz^N_{\fz} u_p=\bdz u_p$ and the identity in Remark \ref{rem-n-2} one has
$${\rm div}(BDw)=0\quad {\rm in}\quad \Omega,$$
where
$$w=\ln |Du_p|,\quad B=(p-2)\frac{Du_p\otimes Du_p}{|Du_p|^2}+I_2.$$
Here  $\otimes$ stands for tensor product and $I_2$ is identity matrix.
Since $Du_p\neq 0$ in $\Omega$, then \eqref{m1} follows from the maximum principle.

To prove \eqref{m2}, for each $1\le j\le 2$ we differentiate equation \eqref{acp} to get
$$\left(|Du_p|^{p-2}\frac{\partial u_p}{\partial x_i}\right)_{x_ix_j}=0
\quad {\rm in}\quad \Omega.$$
This further leads to
$${\rm div}\left(A \frac{\partial u_p}{\partial x_j}\right)=0\quad {\rm in}\quad \Omega,$$
where
$$A=|Du_p|^{p-2}\left[(p-2)\frac{Du_p\otimes Du_p}{|Du_p|^2}+I_2\right].$$
By $Du_p\neq 0$ in $\Omega$, thus \eqref{m2} also holds via the maximum principle.
\end{proof}

Denote by $u$ be the $\fz$-harmonic potential in  $\Omega$.
Recall that $Du_p\to Du$ is locally uniform on $\Omega$ in Theorem \ref{c1-re}, we have the following monotonicity property for $u$.
\begin{lembpp}\label{mn-u}
For all $B(z,r)\Subset \Omega$, we have
\begin{align}\label{mu1}
\max_{\overline {B(z,r)}}|Du|=\max_{\partial B(z,r)}|Du|,\quad \min_{\overline {B(z,r)}}|Du|=
\min_{\partial B(z,r)}|Du|
\end{align}
and
\begin{align}\label{mu2}
\max_{\overline {B(z,r)}}u_{x_i}=\max_{\partial B(z,r)}u_{x_i},\quad \min_{\overline {B(z,r)}}u_{x_i}=
\min_{\partial B(z,r)}u_{x_i} \quad\mbox{ for $ i=1,2$}.
\end{align}
\end{lembpp}

\begin{proof}
We only prove  $\max_{\overline {B(z,r)}}|Du|=\max_{\partial B(z,r)}|Du|$;   the others can be proved in a similar way.
 Assume that is not correct. Then one can find a point $x\in B(z,r)$ such that
 $|Du(x)|>\max_{\overline {B(z,r)}}|Du|$.  Since $|Du_p|\to |Du|  $ in  $C^0(\overline {B(z,r)})$,
 for all sufficiently large $p$, we have
  $|Du_p(x)|>\max_{\overline {B(z,r)}}|Du_p|$, which is a contradiction.
\end{proof}

Due to Lebesgue in \cite{l07}, the monotonicity allows us to get the following.
\begin{lembpp}\label{mn-u2}
For all $B(z,4r)\Subset \Omega$, we have
\begin{align}\label{m-e1}
\frac{{\rm osc}_{B(z,r)}|Du|^2}{r}\le C\left(
\mint{-}_{B_{2r}(z)}|D|Du|^2|^2\,dx\right)^{\frac 12}
\end{align}
and
\begin{align}\label{m-e2}
 \frac{{\rm osc}_{B(z,r)}u_{x_i}}{r}&\le
 \frac{\|\mathcal D[u_{x_i}]\|(B(z,2r))}{r^2} \quad\mbox{ for $ i=1,2$}.
\end{align}

\end{lembpp}

\begin{proof}
Let $B(z,4r)\Subset \Omega$ and $0<\ez <\frac14 r$.
We set $g^{\ez}=|Du|^2\ast \eta^{\ez}$ where
$\eta^{\ez}$ is standard mollifier.
Using polar coordinates in the plane, we have
\begin{align*}
{\rm osc}_{\partial B(z,\rho)}g^{\ez}\le \int^{2\pi}_0 \frac{\partial g^{\ez}}{\partial \tz }(\rho,\tz)\,d\tz.
\end{align*}
Observing
\begin{align*}
|Dg^{\ez}|^2= \rho^{-2}\left(\frac{\partial g^{\ez}}{\partial \tz }\right)^2+
\left(
\frac{\partial g^{\ez}}{\partial \rho }\right)^2\ge \rho^{-2}\left(\frac{\partial g^{\ez}}{\partial \tz }\right)^2,
\end{align*}
we have
\begin{align*}
\int^{2r}_r{\rm osc}_{\partial B(z,\rho)}  g^{\ez}d\rho
\le \int^{2\pi}_0\int^{2r}_r\rho^{-1}\frac{\partial g^{\ez}}{\partial \tz }(\rho,\tz)\rho\,d\rho d\tz
\le\int_{B(z,2r)}|Dg^{\ez}|\,dx.
\end{align*}
Since $|Du|\in C^0_{\loc}(\Omega)\cap W^{1,2}_{\loc}(\Omega)$,
by \cite[Theorem 6, Appendix C.4]{e98} we have
$g^{\ez}\to |Du|^2$ in $C^0(B(z,2r))$ and $Dg^{\ez}\to D|Du|^2$ in
$L^2(B(z,2r))$ as $\ez\to 0$.
Sending
$\ez\to 0$, one gets
\begin{align*}
\int^{2r}_r{\rm osc}_{\partial B(z,\rho)}|Du|^2\,d\rho
\le\int_{B(z,2r)}|D|Du|^2|\,dx.
\end{align*}
Since \eqref{mu1} gives
$$\osc_{B(z,r)} |Du|^2
\le\osc_{\overline {B(z,\rho)}} |Du|^2
=\osc_{\partial B(z,\rho)} |Du|^2,\quad \forall\rho\in[0,2r],$$
we have
$$ \frac{\osc_{B(z,r)} |Du|^2}r\le \frac1{r^2} \int^{2r}_r{\rm osc}_{\partial B(z,\rho)}|Du|^2\,d\rho
\le \frac1r\int_{B(z,2r)}|D|Du|^2|\,dx.$$
Applying the H\"older inequality,   we obtain  \eqref{m-e1} as desired.

 Similarly, for
$1\le i\le 2$, write
  $u^{\ez}_{x_i}=u_{x_i}\ast \eta^{\ez}$ where
$\eta^{\ez}$ is standard mollifier. By an argument similar to  above, one  has
\begin{align*}
\int^{2r}_{r}{\rm osc}_{\partial B(z,\rho)}u^{\ez}_{x_i}\,d\rho
\le \int_{B(z,2r)}|Du^{\ez}_{x_i}|\,dx
\end{align*}
By $u_{x_i}\in BV_{\loc}(\Omega)$ and $u^\ez_{x_i}=u_{x_i}\ast \eta^\ez$, one has
 $$\limsup_{\ez\to 0}\int_{B(z,2r)}|Du^{\ez}_{x_i}|\,dx\le \|\mathcal{D}[u_{x_i}]\|(B(z,2r)).$$
Using $u_{x_i}^\ez\to u_{x_i} $ in $C^0(\Omega)$ again,
$$
\int^{2r}_{r}{\rm osc}_{\partial B(z,\rho)}u_{x_i}\,d\rho
 \le \|\mathcal{D}[u_{x_i}]\|(B(z,2r)).
$$
Since \eqref{mu2} yields
$$\frac{\osc_{B(z,r)} u_{x_i}}r\le  \frac1{r^2}
\int^{2r}_{r}{\rm osc}_{\partial B(z,\rho)}u_{x_i}\,d\rho,$$
we obtain \eqref{m-e2} as desired.
\end{proof}

\section*{Acknowledgments}

The authors would like to thanks Professor Peter Lindqvist for several valuable   comments and suggestions, in particular, pointing out a mistake in earlier version.

F. Peng is supported by the National Natural Science Foundation of China (No.12201612) and the Project funded by China Postdoctoral Science Foundation (No. BX20220328).
Y. Zhang is funded by National Key R\&D Program of China (Grant
No. 2021YFA1003100),  the Chinese Academy of Science, NSFC grant No. 12288201, and CAS Project for Young Scientists in Basic Research, Grant No. YSBR-03.
 Y. Zhou is supported by NSFC (No.12025102) and by the Fundamental Research Funds for the Central Universities.

\noindent  Fa Peng

\noindent
School of Mathematical Sciences, Beihang University, Beijing 100191, P.
R. China

\noindent{\it E-mail }:  \texttt{pengfa@buaa.edu.cn}

\bigskip

\noindent Yi Ru-Ya Zhang

\noindent
Academy of Mathematics and Systems Science, the Chinese Academy of Sciences, Beijing 100190, P. R. China

\noindent{\it E-mail }:  \texttt{yzhang@amss.ac.cn}

\bigskip

\noindent  Yuan Zhou

\noindent
School of Mathematical Science, Beijing Normal University, Haidian District Xinjiekou Waidajie No.19, Beijing 10875, P. R. China

\noindent{\it E-mail }:  \texttt{yuan.zhou@bnu.edu.cn}
\end{document}